\documentclass{amsart}
\usepackage{graphicx}
\usepackage{amsmath,amssymb}
\usepackage{empheq}
\usepackage{color, tikz}
\usepackage[utf8]{inputenc}

\newtheorem{thm}{Theorem}[section]

\newtheorem{lem}[thm]{Lemma}
\newtheorem{prop}[thm]{Proposition}

\theoremstyle{definition}

\theoremstyle{remark}
\newtheorem{rem}[thm]{Remark}

\numberwithin{equation}{section}

\renewcommand{\parallel}{\mathrel{/\mkern-5mu/}}
\makeatletter
\newcommand{\notparallel}{%
  \mathrel{\mathpalette\not@parallel\relax}%
}
\newcommand{\not@parallel}[2]{%
  \ooalign{\reflectbox{$\m@th#1\smallsetminus$}\cr\hfil$\m@th#1\parallel$\cr}%
}
\makeatother

 \theoremstyle{plain}

\newcommand{\lam}{\lambda}

\newcommand{\Del}{\Delta}

\newcommand{\calF}{{\mathcal F}}
\newcommand{\calH}{{\mathcal H}}
\newcommand{\calI}{{\mathcal I}}

\begin{document}

\title[]{nonsymmetric traveling wave solution to a Hele-Shaw type tumor growth model}%

\author{Yu Feng, Qingyou He, Jian-Guo Liu, Zhennan Zhou}%

\address{Yu Feng: School of Science, Great Bay University, Guangdong, 523000, China}
\email{fengyu@gbu.edu.cn}
\address{Qingyou He: Sorbonne Universit{\'e}, CNRS, Laboratoire de Biologie Computationnelle et Quantitative, F-75005 Paris, France}
\email{qyhe.cnu.math@qq.com}
\address{Jian-Guo Liu: Department of Mathematics and Department of
  Physics, Duke University, Durham, NC 27707, USA}
\email{jian-guo.liu@duke.edu}
\address{Zhennan Zhou: Institute for Theoretical Sciences, Westlake University, Hangzhou, Zhejiang Province, 310030, China}
\email{zhouzhennan@westlake.edu.cn}
\date{\today}


\thanks{}

\begin{abstract}
We consider a Hele-Shaw model that describes tumor growth subject to nutrient supply. The model is derived by taking the incompressible limit of porous medium type equations, and the boundary instability of this model was recently studied in \cite{feng2022tumor} using asymptotic analysis. In this paper, we further prove the existence of nonsymmetric traveling wave solutions to the model in a two dimensional tube-like domain, which reflect intrinsic boundary instability in tumor growth dynamics.
\end{abstract}


\maketitle


\section{introduction}
\label{sec:model introduction}
Tumors remain one of the most severe diseases threatening human life and health. The instability of the tumor boundary, characterized by the formation and evolution of finger-like protrusions, plays a crucial role in tumor invasion and metastasis. These structures enable tumor cells to infiltrate surrounding healthy tissues more efficiently, contributing to the aggressiveness of malignant tumors \cite{amar2011contour, cristini2005morphologic, ye2024fingering}. Consequently, understanding the mechanisms underlying these morphological instabilities is of both theoretical and clinical significance. This paper investigates these mechanisms within the framework of a mathematical model that has gained increasing attention in recent years \cite{DAVID2021,DPSV2021,GKM2022,KM2022,PERTHAME2014}. To highlight the relevance of this model, we first review the development and classification of mathematical tumor growth models.

Mathematical modeling has become a powerful tool for studying tumor growth dynamics, offering insights into the fundamental mechanisms that drive morphological instabilities \cite{borisovich2005symmetry, cristini2003nonlinear, friedman2007bifurcation, gatenby1996reaction}. Existing tumor growth models can be broadly classified into two main categories. The first category consists of reaction-diffusion models, which describe the spatiotemporal evolution of tumor cell density using parabolic-type partial differential equations (PDEs) incorporating proliferation, diffusion, and chemotaxis \cite{oden2016toward,sfakianakis2020mathematical}. These models effectively capture the macroscopic dynamics of tumor expansion and have been extensively studied in mathematical oncology. The second category includes Hele-Shaw-type free boundary models, which were originally developed to describe the flow of a viscous fluid confined between two closely spaced parallel plates and were first applied to model tumor growth by Greenspan in \cite{greenspan1972models, greenspan1976growth}. In this framework, the tumor is treated as a saturated domain with a moving free boundary, where the internal pressure determines the boundary evolution through Darcy's law. These models have been particularly useful for studying the formation of finger-like structures observed in tumor invasion.

A fundamental challenge in the mathematical study of Hele-Shaw-type tumor growth models is the analysis of their well-posedness and boundary instability. Unlike Stefan problems \cite{friedman1968stefan}, parabolic-type free boundary problems, Hele-Shaw-type models lack time regularization due to their elliptic nature. Although local existence results have been established \cite{chen2003free}, its long-time general well-posedness remains an open problem. In the past two decades, two major advances have been made in addressing the mathematical challenges of these models. For one thing, Friedman and his collaborators studied a class of Hele-Shaw-type tumor growth models by perturbing radially symmetric solutions, establishing the existence of nonsymmetric steady-state solutions with simple-mode perturbations \cite{borisovich2005symmetry, friedman353symmetry}. Their results provided key insights into the mechanisms of boundary instability. For another thing, Perthame et al. rigorously derived a class of Hele-Shaw-type free boundary problems as the incompressible limit of porous medium equations (PME), bridging the gap between reaction-diffusion models and Hele-Shaw models \cite{PERTHAME2014}. This connection offers a unified framework linking two widely used modeling paradigms. In the following, we refer to such incompressible-imit models as the PME-derived Hele-Shaw models.

Although the PME-derived Hele-Shaw models share some similarities with those initially proposed by Greenspan \cite{greenspan1972models, greenspan1976growth} and further developed by Friedman et al. \cite{chen2003free, cui2001analysis, friedman2006bifurcation, friedman2007bifurcation, friedman2008stability, friedman1999analysis}, they exhibit significant differences. Notably, in the Greenspan-type models, the pressure can take negative values and the boundary conditions depend on the curvature of the interface. Moreover, in these models, the coefficient in front of the curvature term usually serves as the bifurcation parameter in boundary instability studies \cite{byrne1996growth, cui2006formation, cui2001analysis, hao2012bifurcation, wu2021analysis}. In contrast, in the PME-derived models, the pressure, as the limit of a density power function, must remain nonnegative and vanish at the tumor boundary. For a detailed discussion of these differences, we refer the reader to \cite{dou2023tumor}. Importantly, while significant progress has been made in analyzing boundary instabilities in Greenspan-type models, the instability properties of PME-derived Hele-Shaw models remain largely unexplored.

The objective of this paper is to investigate boundary instability in a PME-derived Hele-Shaw type tumor growth model in dimension two. The model under consideration, first proposed in \cite{perthame2014traveling}, is given by
\begin{subequations}
\label{eqn: mainmodel}
\begin{alignat}{2}
\label{eqn: pressure0}
-\Del p &= G_0 c,&\qquad\text{in}\quad D(t),\\
\label{bc: pressure0}
p&=0,&\qquad\text{on}\quad\partial D(t),\\
-\Delta c+\lambda c&=0,&\qquad\text{in}\quad D(t),\\
-\Delta c+ c&=c_B,&\qquad\text{in}\quad \mathbb{R}^2\setminus D(t),
\end{alignat}
where $G_0, c_B, \lam>0$, and $\textbf{x}\in D(t)\subseteq\mathbb{R}^2$ represents the tumor domain at time $t$, $p(\textbf{x},t)$ is the internal pressure, and $c(\textbf{x},t)$ denotes the nutrient concentration. The boundary evolution follows Darcy's law:
\begin{equation}
\label{bv}
    v\vert_{\partial D}=-\nabla p\cdot n\vert_{\partial D},
\end{equation} 
\end{subequations}
where $n$ stands for the outer normal vector on the free boundary. The model \eqref{eqn: mainmodel} can be interpreted as follows: $G_0 c$ is the growth rate function. Within the tumor, nutrients are consumed by tumor cells at a rate of $\lambda$, while outside the tumor, nutrients are supplied by the vascular network in the healthy region, with the supply rate proportional to the concentration difference $c_B-c$ for the density of the background nutrients $c_B$.

A recent study by Feng et al. \cite{feng2022tumor} investigated the boundary instability of \eqref{eqn: mainmodel} using an asymptotic analysis approach, which complements the current understanding of this model \cite{feng2024unified,jiang2025efficient, liu2018accurate, liu2018analysis, perthame2014traveling}. They initialize a small perturbation to \eqref{eqn: mainmodel} around the symmetric solutions with perturbation profiles given by cosine functions and perturbation amplitude $\epsilon(t)$ small enough, which reduced the evolution of the boundary perturbation to the dynamics of the perturbation amplitude, in particular, they derived the so-called boundary evolution equation (see \cite{cristini2003nonlinear}):
$$
\epsilon^{-1}\frac{d\epsilon}{dt}=E(\lam,l),
$$ 
and characterized the boundary instability by determining the sign of $E(\lam,l)$. When it takes a positive value, the finger-like structures grow; otherwise, the boundary degenerates to the symmetric one. The main result in \cite{feng2022tumor} interprets that the nutrient consumption rate $\lambda$ can trigger the boundary instability in \eqref{eqn: mainmodel}. Specifically, the boundary remains stable for any perturbation frequency if $\lambda \le 1$. However, for $\lambda>1$, there is a threshold value for the perturbation frequency, below which the boundary instability occurs, although higher frequencies remain stable. Recently, the boundary instability for the same model in dimension three has been explored using asymptotic analysis in \cite{liu2024three}.

In this paper, we further reveal the intrinsic boundary instability of model~\eqref{eqn: mainmodel} by proving the existence of nonsymmetric traveling wave solutions in a tube-like domain. The proof is based on linking such traveling wave solutions to the nontrivial bifurcation branches to a nonlinear functional equation 
$$F(\xi,\lam)=0,$$ 
where $\xi$ is a function that describes the boundary profile and $\lam$ is the parameter of nutrient consumption rate in the model. The symmetric traveling wave solutions of \eqref{eqn: mainmodel} are naturally linked to the branch of trivial solutions $(0,\lam)$. Our main result demonstrates that the above functional equation admits nontrivial solutions (for which $\xi\neq 0$), and such solutions can be interpreted as a curve parameterized by $\epsilon$ in the $(\xi,\lam)$ phase plane. We summarize it as follows (see Theorem \ref{thm: mainthm} for a more precise version):

\begin{thm} For each perturbation frequency $l\in \mathbb{N}$, there exists a bifurcation point $\lam_0^l$ such that one can find a bifurcation branch $(\xi_{\epsilon}(y),\lam_\epsilon)$ (parameterized by the small parameter $0<\epsilon\ll 1$) starting from $(0,\lam_0^l)$ such that $F(\xi_\epsilon(y),\lam_\epsilon)=0$. Thus, there exist nontrivial traveling wave solutions to the model with the boundary profile given by $\xi_\epsilon(y)$ associated with nutrient consumption rate $\lam_\epsilon$.
\end{thm}

The proof of the theorem is carried out by investigating the Fr\'{e}chet derivative of the nonlinear map $F(\xi,\lam)$ on the line $(0,\lam)$, denoted as $F_\xi(0,\lam)$. In particular, we show that $F_\xi(0,\lam)$ as a bounded linear operator can be characterized in terms of an eigenvalue problem (see Proposition \ref{lem:Fder}, Remark \ref{rem:linearity}, and Remark \ref{rem:eproblem}), with the complete basis of eigenfunctions given by cosine functions, and the associated eigenvalues coincide with $E(\lam,l)$. Finally, utilizing the explicit expression of $E(\lam,l)$, we conclude the proof by determining the bifurcation points and verifying the assumptions of the celebrated Crandall-Rabinowitz theorem (Theorem~\ref{thm:Crandall}).

As discussed above, model~\eqref{eqn: mainmodel} is derived from taking the incompressible limit of the PME type model (see \eqref{eqn:pme density eqn}-\eqref{eqn:finite m invivo} below). Therefore, we provide a detailed literature review of incompressible limit studies in tumor growth models. Perthame et al. first generalized related studies to tumor growth models in the seminal work~\cite{PERTHAME2014}, which facilitates numerous impressive works in this direction ~\cite{DAVID2021,DPSV2021,DS2020,GKM2022,KM2022,KT2018,LX2021}. The Hele-Shaw asymptotic limit of the tumor growth model~\cite{DAVID2021,PQTV2014, PERTHAME2014,PV2015} was initially studied. 
For the tumor growth model with Brinkman's pressure law governing the motion, the authors in~\cite{KT2018} established an optimal uniform convergence of the density and the pressure, where the optimality says the convergence in $L_{loc}^\infty$ norm. The two-species case's Hele-Shaw (incompressible) limits were proved in~\cite{DPSV2021,DS2020}.  The Hele-Shaw limit of the PME with the non-monotonic (even non-local) reaction terms through the approach of the obstacle problem was completed in~\cite{GKM2022,KM2022}. The existence of weak solutions and the free boundary limit of a tissue growth model with autophagy were obtained in~\cite{LX2021}, respectively.  The singular limit of the PME with a drift was discussed in~\cite{KPW2019}. Recently, the convergence of free boundaries in the incompressible limit of tumor growth models was considered in \cite{JY2024}. In addition, the incompressible limits for chemotaxis, even with growth effect, were shown in \cite{CKY2018,HLP2023,HLP2022}. 

Before concluding, we draw the reader’s attention to recent works \cite{jacobs2023tumor,kim2023tumor}, which investigate the boundary behavior of models closely related to \eqref{eqn: mainmodel}, but with nutrient dynamics governed by a parabolic type equation. In \cite{jacobs2023tumor}, the authors establish boundary regularity for the coupled model and further demonstrate that when nutrient diffusion is absent and cell death is neglected, the tumor density exhibits a regularizing effect. Based on this, \cite{kim2023tumor} shows that if nutrient diffusion is small, the tumor patch boundary remains within a small neighborhood of the smooth boundary obtained in the nondiffusive case studied in \cite{jacobs2023tumor}.

\vspace{1mm}

\noindent\textbf{Organization of this paper.}   In Section \ref{sec:Incompressible limit derivation}, we provide a formal derivation of the Hele-Shaw model by taking the incompressible limit of a density model in the type of porous medium equation. In Section \ref{sec:Bifurcation Analysis}, we prove the existence of nonsymmetric traveling wave solutions to this model in a tube-like domain by using the Crandall-Rabinowitz Theorem.

\section{Model derivation}
\label{sec:Incompressible limit derivation}
This section is devoted to providing a formal derivation of the model \eqref{eqn: mainmodel} by taking the incompressible limit of a cell density model of the PME type. We begin by introducing this cell density model in the following subsection.
\subsection{A cell density model of PME type}
We let $\rho_m(\textbf{x},t)$ to denote the tumor cell density, start from the initial state $\rho_{m,0}(\textbf{x})$, and $D_m(t)$ be the supporting set of $\rho_m(\textbf{x},t)$ that is 
\begin{equation}
\label{eqn:set D}
    D_m(t)=\left\{\textbf{x}\vert\rho_m(\textbf{x},t)>0\right\}.
\end{equation}
Physically, it presents the tumoral region at time $t$. We assume the pressure inside the tumor follows the constitutive relationship of 
\begin{equation}
 P_m(\rho_m)=\frac{m}{m-1}\rho_m^{m-1}, 
\end{equation}
and the tumor cell moving velocity is governed by the Darcy's law $V=-\nabla P_m$. Thus, in particular, the boundary expansion is characterized by $V\vert_{\partial D_m}=-\nabla P_m\vert_{\partial D_m}$ and $D_m(t)$ remains finite for any $t<\infty$ provided $D(0)$ is compact. On the other hand, we employ $c_m(\textbf{x},t)$ to denote the nutrient concentration and assume the cells' production rate is proportional to it. With the above assumption, the evolution of $\rho_m$ satisfies the following porous medium equation (PME) with source term:
\begin{equation}
\label{eqn:pme density eqn}
    \partial_t\rho_m-\nabla\cdot\left(\rho_m\nabla P_m\right)=G_0 c_m\rho_m,\quad t\geqslant 0, \quad G_0>0.
\end{equation}
Regarding the nutrient concentration $c_m$, it is governed by the following reaction-diffusion equation in general:
\begin{equation}
\label{eqn:parabolic nutrient}
    \tau \partial_t c_m-\Delta c_m +\Psi(\rho_m,c_m)=0,
\end{equation}
where the parameter $\tau\geq 0$ characterizes the nutrient change time scale, and the binary function $\Psi(\rho_m,c_m)$ describes the overall effects of the nutrient supply outside the tumor and the nutrient consumption inside the tumor. Considering the timescale parameter $\tau\ll 1$ (see, e.g. \cite{adam1997survey, byrne1996growth}), we drop it to get the elliptical form of nutrient equation,
\begin{equation}
\label{eqn:general nutrient eqn}
    -\Delta c_m +\Psi(\rho_m, c_m)=0.
\end{equation}
Then, we focus on the so-called \emph{in vivo} regime in \cite{feng2022tumor}, in which the nutrients are provided by vessels of the healthy tissue surrounding the tumor. Mathematically, we assume the nutrients is consumed at a rate of $\lam>1$ in the tumor cell saturated region $S_m(t)=\{\textbf{x}\vert \rho_m(\textbf{x},t)\geq 1\}$. While, outside $S_m(t)$, the nutrient supply is determined by the concentration difference from the background, $c_B - c_m$, with the rate of $(1-\rho_m)_+$. Thus, the binary function takes the form of
\begin{equation}
\label{eqn:PME_overall}
    \Psi(\rho_m,c_m)=\lambda\rho_m c_m\cdot\chi_{S_m}-(1-\rho_m)_+(c_B-c_m)\cdot\chi_{S_m^c},
\end{equation}
where $f_+=\max\{f,0\}$. Or equivalently,
\begin{subequations}
\label{eqn:finite m invivo}
\begin{alignat}{2}
-\Delta c_m + \lambda\rho_m c_m= 0,\qquad &\text{for}\quad x\in S_m(t),\\
-\Delta c_m =(1-\rho_m)_+(c_B-c_m),\qquad &\text{for}\quad x\in \mathbb{R}^n\setminus S_m(t).
\end{alignat}
\end{subequations}

By now, we have finished introducing the cell density model (\eqref{eqn:pme density eqn} coupled with \eqref{eqn:PME_overall}, equivalently, \eqref{eqn:finite m invivo}). 

\subsection{A formal derivation}
In this subsection we derive \eqref{eqn: mainmodel} formally. Firstly, by multiplying $m\rho_m^{m-1}$ on the both sides of equation \eqref{eqn:pme density eqn} one gets the equation for pressure
\begin{equation}
\label{eqn: finite_pressure}
    \partial_t P_m=\left\vert\nabla P_m\right\vert^2+(m-1) P_m\left(\Delta P_m+G_0 c_m\right).
\end{equation}
Then, by sending $m\rightarrow\infty$ and let $(\rho_{\infty}, P_{\infty}, c_{\infty})$ denote the limit density, pressure, and nutrient, respectively. One can formally obtain the so-called \emph{complementarity condition}:
\begin{equation}
\label{eqn:complementary}
    P_{\infty}(\Delta P_{\infty}+G_0 c_\infty)=0.
\end{equation}
At the same time, the limit density $\rho_{\infty}$ satisfies the following equation in the distributional sense:
\begin{equation}
       \frac{\partial}{\partial t}\rho_{\infty}-\nabla\cdot\left(\rho_{\infty}\nabla P_{\infty}\right)=G_0 \rho_{\infty} c_{\infty},
\end{equation}
with $P_{\infty}$ compels $\rho_{\infty}$ only take value in the range of $\left[0,1\right]$ for any initial date $\rho_{\infty,0}\in\left[0,1\right]$. And $P_{\infty}$ belongs to the Hele-Shaw monotone graph:
\begin{equation}
\label{eqn:HS graph}
P_{\infty}(\rho_{\infty})=\left\{
  \begin{array}{rcr}
    0, \qquad 0\leqslant\rho_{\infty}<1,\\
    \left[0,\infty\right),\qquad \rho_{\infty}=1.\\
  \end{array}
\right.
\end{equation}
Observe that, in general, the supporting set of $\rho_{\infty}$ is larger than that of $P_{\infty}$. However, a transparent regime called "patch solutions" exists for a large class of initial data, in which the two sets coincide with each other, denoted by $D_{\infty}(t)$. At the same time, the limit density evolves in the form of $\rho_{\infty}=\chi_{D_{\infty}(t)}$, where $\chi_{A}$ presents the characteristic function of the set $A$. In this specific regime, \eqref{eqn:finite m invivo} reduce to
\begin{subequations}
\label{eqn: infinite m invivo}
\begin{alignat}{2}
-\Delta c_\infty + \lambda c_\infty= 0,\qquad &\text{for}\quad x\in D_{\infty}(t),\\
-\Delta c_\infty =c_B-c_\infty,\qquad &\text{for}\quad x\in \mathbb{R}^n\setminus D_{\infty}(t).
\end{alignat}
\end{subequations}
While, \eqref{eqn:complementary} and \eqref{eqn:HS graph} together yields
\begin{subequations}
\label{eqn: limitpressureeqn}
\begin{alignat}{2}
-\Del P_\infty &= G_0 c_\infty,&\qquad\text{in}\quad D_{\infty}(t),\\
P_\infty&=0,&\qquad\text{on}\quad\partial D_{\infty}(t).
\end{alignat}
\end{subequations}
Finally, by dropping the subscripts in \eqref{eqn: infinite m invivo} and \eqref{eqn: limitpressureeqn}, and replacing $P$ by $p$ in \eqref{eqn: limitpressureeqn}, one gets the desired model \eqref{eqn: mainmodel}. We emphasize that before taking the incompressible limit, the pressure function \eqref{eqn: finite_pressure} and the nutrient functions \eqref{eqn:finite m invivo} are coupled strongly to each other. However, fortunately, by taking the incompressible limit and further restricting ourselves to the patch solution regime, the two equations decoupled automatically. 

The above derivation can be presented rigorously. For the reader's convenience, we provide a complete proof in Appendix~\ref{ic}.

\section{Bifurcation Analysis}
\label{sec:Bifurcation Analysis}
This section is dedicated to proving the existence of nonsymmetric traveling wave solutions for \eqref{eqn: mainmodel} in a tube-like domain. We begin by introducing the necessary notations in Section \ref{sec:Model set up}, followed by solving symmetric solutions in Section \ref{sec:Symmetric solutions}. In Section \ref{sec:Dissucssion on non-symmtric solutions}, we describe the nonsymmetric solutions obtained by perturbing symmetric ones and outline our proof strategy. Finally, Sections \ref{sec:The linearized system} through \ref{sec:Existence of nontrivial bifurcation branches} present the detailed proof, following the framework established at the end of Section \ref{sec:Dissucssion on non-symmtric solutions}.

\subsection{Model set up}
\label{sec:Model set up}
To begin with, let $\Omega$ denote the entire tube-like domain defined as follows,
\begin{equation}
\label{eqn:infinity_tube}
\Omega = \left\{(x,y) \mid (x,y) \in (-\infty,+\infty) \times [-\pi,\pi]\right\},
\end{equation}
and we post periodic boundary conditions for the $y$ variable. For any $2\pi$ periodic even function $\xi(y)$, the tumor region is defined by
\begin{equation}
\label{eqn: perturbed_domain_X1}
   \Omega_{\xi}=\left\{(x,y)\vert x\leq \xi(y),\quad y\in[-\pi,\pi]\right\},
\end{equation}
and the associated boundary is given by
\begin{equation}
\label{eqn: perturbed_boundary_X1}
   \mathcal{B}_{\xi}=\left\{(x,y)\vert x=\xi(y),\quad y\in[-\pi,\pi]\right\}.
\end{equation}
We employ the superscripts (i) and (o) to denote the solutions inside and outside the tumor region, respectively. Inside the tumor region, the nutrient and pressure $(c^{\text{(i)}}(x,y;\xi),p^{\text{(i)}}(x,y;\xi))$ satisfy that 
\begin{subequations}
\label{eqn:general_all}
\begin{alignat}{2}
\label{eqn:general_inner}
    -\Delta c^{\text{(i)}}+\lam c^{\text{(i)}}&=0,&\quad\text{in}\quad \Omega_{\xi},\\
    -\Delta p^{\text{(i)}}&=G_0 c^{\text{(i)}},&\quad\text{in}\quad \Omega_{\xi}.
\end{alignat}
In contrast, outside the tumor region, we require  $(c^{\text{(o)}}(x,y;\xi),p^{\text{(o)}}(x,y;\xi))$ to satisfy 
\begin{alignat}{2}
\label{eqn:general_outer}
    -\Delta c^{\text{(o)}}+c^{\text{(o)}}&=c_B,&\quad\text{in}\quad \Omega\setminus\Omega_{\xi},\\
    -\Delta p^{\text{(o)}}&=G_0 c^{\text{(o)}},&\quad\text{at}\quad \Omega\setminus\Omega_{\xi}.
\end{alignat}
The inner solutions and the outer solutions are matched at the boundary in the following sense
\begin{alignat}{2}
\label{bc:general}
     c^{\text{(i)}}&=c^{\text{(o)}},&\quad\text{at}\quad \mathcal{B}_{\xi},\\
     \frac{\partial}{\partial n}c^{\text{(i)}}&= \frac{\partial}{\partial n}c^{\text{(o)}},&\quad\text{at}\quad \mathcal{B}_{\xi},\\
    p^{\text{(i)}}&=p^{\text{(o)}}=0,&\quad\text{at}\quad \mathcal{B}_{\xi},
\end{alignat}
where $\frac{\partial}{\partial n}$ represents the outer normal derivative of the tumor region $\Omega_{\xi}$ on the boundary $\mathcal{B}_{\xi}$.
Besides, we also require that
\begin{alignat}{2}
     c^{\text{(i)}}&<\infty,&\quad\text{at}\quad x= -\infty,\\
     c^{\text{(o)}}&=c_B&\quad\text{at}\quad x= +\infty,\\
     p^{\text{(i)}}&<\infty,&\quad\text{at}\quad x= -\infty.
\end{alignat}
\end{subequations}
Note that the pressure is extended to the entire domain only for technical purposes and, thus, we do not require $p^{\text{(o)}}(x,y;\xi)$ to be bounded at $x=+\infty$. And, for notation simplicity we introduce:
\begin{equation}
\label{eqn:general_wholesoln}
    c(x,y;\xi):=c^{\text{(i)}}(x,y;\xi)+c^{\text{(o)}}(x,y;\xi),\quad p(x,y;\xi):=p^{\text{(i)}}(x,y;\xi)+p^{\text{(o)}}(x,y;\xi).
\end{equation}
We emphasize that for arbitrary $(\xi(y),\lam)$, the boundary profile is not steady in general; it tends to evolve according to Darcy's law, and the normal speed $v$ on the boundary is given by
\begin{equation}
    v\vert_{\mathcal{B}_\xi}=-\nabla p^{\text{(i)}}\cdot n\vert_{\mathcal{B}_\xi},
\end{equation}
where $n$ stands for the outer normal vector on $\mathcal{B}_\xi$.

Our goal is to demonstrate that perturbing the symmetric solutions, where $\xi(y)\equiv 0$, of equation \eqref{eqn:general_all} can result in nontrivial solutions, where the boundary profile 
$\xi(y)\neq 0$ remains steady and propagates to the right at a constant speed. To achieve this, the next subsection is dedicated to solving for symmetric solutions.

\subsection{Symmetric solutions}
\label{sec:Symmetric solutions}
For symmetric solutions, the tumor occupies the left half of the tube and we denote it by $\Omega_0$ as follows
\begin{equation}
\label{eqn:symmetric_domain}
\Omega_0=\left\{(x,y)\vert x\leq 0\right\},
\end{equation}
and the corresponding boundary is given by
\begin{equation}
\mathcal{B}_0=\left\{(x,y)\vert x= 0\right\}.
\end{equation}
Provide any $\lam>0$, by solving \eqref{eqn:general_all} but with $\Omega_\xi$ and $\mathcal{B}_{\xi}$ replaced by $\Omega_0$ and $\mathcal{B}_0$, respectively, we get symmetric solutions as follows:
\begin{subequations}
\label{eqn: invivo_unperturbed_solns}
\begin{alignat}{2}
&c_0^{\text{(i)}}(x)=\frac{c_B}{\sqrt{\lam}+1}e^{\sqrt{\lam}x},&\quad \text{for}\quad x\leq 0,\\
&c_0^{\text{(o)}}(x)=c_B-\frac{c_B\cdot \sqrt{\lam}}{\sqrt{\lam}+1}e^{-x},&\quad \text{for}\quad x\geq 0,\\
&p_0^{\text{(i)}}(x)=\frac{G_0\cdot c_B}{\lam(\sqrt{\lam}+1)}-\frac{G_0\cdot c_B}{\lam(\sqrt{\lam}+1)}e^{\sqrt{\lam}x},&\quad \text{for}\quad x\leq 0,\\
&p_0^{\text{(o)}}(x)=G_0 c_B\left(-\frac{x^2}{2}+\frac{\sqrt{\lambda}}{\sqrt{\lambda}+1}e^{-x}+\frac{\lambda-1}{\sqrt{\lambda}(\sqrt{\lambda}+1)}x-\frac{\sqrt{\lambda}}{\sqrt{\lambda}+1}\right),&\quad \text{for}\quad x\geq 0,
\end{alignat}
and the associated traveling speed is given by
\begin{equation}
\label{eqn: invivo_v0}
    v_0(\lambda)=-\frac{\partial p_0^{\text{(i)}}}{\partial x}(0)=\frac{G_0\cdot c_B}{\sqrt{\lam}(\sqrt{\lam}+1)}.
\end{equation}
\end{subequations}
Finally, we define $(c_0(x),p_0(x))$ in the same way as \eqref{eqn:general_wholesoln}.
\subsection{The perturbed system}
\label{sec:Dissucssion on non-symmtric solutions}
This subsection aims to clarify the type of nonsymmetric solution we are looking for and introduce the corresponding notations.

In the previous subsection, we demonstrated that the Hele-Shaw problem can be easily solved under the symmetric assumption. In the subsequent subsections, we aim to show that by perturbing the vertical boundary with profiles possessing a certain symmetry, such as \( \cos{ly} \) (with \( l \in \mathbb{N} \)), and a small amplitude \( 0 < \epsilon \ll 1 \), nonsymmetric traveling wave solutions can be found near a specific nutrient consumption rate \( \lambda_0^l \) (which depends on \( l \)). These solutions satisfy the boundary value problem \eqref{eqn:general_all}, where the boundary profile is approximately \( \epsilon \cos{ly} \).

We aim to establish the existence of a perturbed tumor region \( \Omega_{\epsilon} \) with its boundary \( \mathcal{B}_\epsilon \), along with a corresponding consumption rate \( \lambda_{\epsilon} \), such that the solution to \eqref{eqn:general_all} under \( (\Omega_\epsilon, \mathcal{B}_\epsilon, \lambda_\epsilon) \) remains a steady traveling wave propagating to the right at a constant speed \( v_0(\lambda_\epsilon) \) (recall \eqref{eqn: invivo_v0}). We denote this solution by $(c^{\text{(i)}}_\epsilon, c^{\text{(o)}}_\epsilon, p^{\text{(i)}}_\epsilon, p^{\text{(o)}}_\epsilon)$, and define $(c_\epsilon, p_\epsilon)$ in the same way as \eqref{eqn:general_wholesoln}.

In fact, such nontrivial solution $(\xi_\epsilon,\lambda_\epsilon,c_\epsilon, p_\epsilon)$ can be constructed explicitly in terms of an elegant infinite series solution by using the boundary transformation technique in \cite{friedman353symmetry}. However, the construction and proof of convergence of such infinite series are very complicated. To avoid this tedious procedure, but justify the existence of such a nonsymmetric traveling wave solution, we adopt the Crandall-Rabinowitz framework proposed in \cite{borisovich2005symmetry}. 

The Crandall-Rabinowitz theorem is developed to study the bifurcation behavior in nonlinear equations. It provides conditions under which solutions branch off from trivial solutions in nonlinear functional equations. As an effective analysis tool, the Crandall-Rabinowitz theorem has been widely applied to study the existence and stability of solutions in nonlinear systems, see, e.g., \cite{borisovich2005symmetry, HMV2013, Diaz1987, Wei1997}. In particular, the framework of utilizing the Crandall–Rabinowitz theorem to study the bifurcation behavior of free boundary models was initially proposed by Friedman in \cite{borisovich2005symmetry}, then extensively employed to study the bifurcation phenomenon in different tumor growth models, see \cite{borisovich2005symmetry, friedman2007bifurcation, lu2022bifurcation, zhao2020symmetry,zhao2025determination}.

The advantage of the Crandall-Rabinowitz theorem is that instead of constructing all the terms inductively in the infinity series, we only need to investigate the linear part of the perturbation problem. More specifically, we construct a nonlinear functional map (see \eqref{eqn: bifurcation_fcnal} for a more precise version):
    \begin{equation*}
        F: X\times \mathbb{R} \to Y,\quad (\xi,\lambda) \mapsto F(\xi,\lambda)
    \end{equation*}
based on the leading order of the perturbation problem. Here, the periodic even function \( \xi(y)\) represents the boundary profile, and the parameter \( \lambda > 0 \) represents the nutrient consumption rate in the model. Moreover, $X$ and $Y$ are the function spaces for $\xi$ and $F(\xi,\lam)$ that will be specified more clearly later. The map ensures that the symmetric solution \( (0, \lambda) \) corresponds to the trivial solution of this map, that is, \( F(0, \lambda) = 0 \) for all positive $\lam$. While, the nonsymmetric solution 
\begin{equation*}
F(\xi_{\epsilon}, \lambda_{\epsilon}) = 0,
\end{equation*}
which we seek, corresponds to a bifurcation branch parameterized by the perturbation amplitude $\epsilon$. That is, it shall be viewed as a curve in the phase plane $(\xi,\lam)$ originating from some bifurcation point $(0, \lambda_0)$ with a tangent vector denoted by $(\tau(y), \tau_0)$ for now. Equivalently,
\begin{equation}
\label{eqn:nontriviasoln}
\xi_\epsilon(y) = \epsilon\tau(y) + O(\epsilon^2), \quad \lambda_{\epsilon} = \lambda_0+\epsilon\tau_0+O(\epsilon^2).
\end{equation}
It should be noted that to ensure the existence of such a branch, a necessary condition is that   
\begin{equation}
\label{eqn:totalderivative1} 
    \frac{d}{d\epsilon}F(\xi_\epsilon,\lam_\epsilon)\Big\vert_{\epsilon=0}=\frac{\partial F}{\partial \xi}(0,\lam_0)\tau(y)+\frac{\partial F}{\partial\lam}(0,\lam_0)\tau_0=0.
\end{equation}
Otherwise, the Implicit Function Theorem can be applied by viewing $\epsilon$ as the independent variable to show that the trivial (symmetric) solution has to be the unique solution. Observe the fact that $F(0,\lam)=0$ for all $\lam>0$ yields $\frac{\partial F}{\partial \lam}(0,\lam_0)\equiv 0$. Thus, condition \eqref{eqn:totalderivative1} reduces to
\begin{equation}
\label{eqn:totalderivative2} 
\frac{\partial F}{\partial \xi}(0,\lam_0)\tau(y)=0.
\end{equation}
The Crandall-Rabinowitz Theorem provides conditions on $\frac{\partial F}{\partial \xi}(0,\lam_0)$ such that ensure the existence of the desired branch. The main steps for applying the Crandall-Rabinowitz Theorem are as follows:
    \begin{enumerate}
    \item Fix any $\lam>0$, determine the Fr\'echet derivative of $F(\xi,\lambda)$ with respect to $\xi\in X$ at $(0,\lam)$, denote it as $F_{\xi}(0,\lambda):=\frac{\partial}{\partial\xi}F(0,\lambda)$; and show that the Fr\'echet derivative $F_{\xi}(0,\lambda)$, as a bounded linear operator that maps $X$ to $Y$, can be characterized in terms of an eigenvalue problem  (see equation \eqref{eqn: eigenprob}), with the complete basis of eigenfunctions are given by $\left\{\cos{ly}\right\}_{l=1}^{\infty }$ (the Shauder basis of $X$), and associated eigenvalues, given by $E(\lam,l)$ (see \eqref{eqn:E} for the expression), are distinct, non-degenerate, and real. 
    \item According to the last step, for any $\lam>0$ and when $\tau(y)$ coincide to a single basis function $\cos{ly}$, we have
    \begin{equation*}
        F_\xi(0,\lam)\cos{ly}=E(\lam,l)\cos{ly}.
    \end{equation*}
    Thus, by utilizing the explicit expression of the eigenvalues, for each perturbation frequency $l\in\mathbb{N}$, one can determine a specific consumption rate $\lambda_0^l>l^2$ (see Proposition \ref{rem:F1}) such that $E(\lam_0^l,l)=0$, so that \eqref{eqn:totalderivative2} holds. 
    \item To conclude that $\left\{(0,\lambda_0^l)\right\}_{l=1}^{\infty}$ are indeed bifurcation points that give rise to the desired nontrivial branch of the functional equation $F(\xi,\lambda) = 0$, we verify that the Fréchet derivative at these points, $F_{\xi}(0,\lambda_0^l)$, satisfies the assumptions in the Crandall–Rabinowitz theorem. Roughly speaking, the way of choosing $\lam_0^l$ ensures that $E(\lam_0^l,k)=0$ if and only if $k=l$.
\end{enumerate}
We establish the details of the above procedures in the following subsections.

\subsection{The linearized system}
\label{sec:The linearized system}
We devote this subsection to studying the linearization of \eqref{eqn:general_all} under a small perturbation. More specifically, we expand the solution in terms of the perturbation amplitude $0<\epsilon\ll 1$, and derive the equations with boundary conditions for the first-order terms. 

To begin with, we fix arbitrary $\lam>0$ and take $\xi(y)=\epsilon\tilde{\xi}(y)$, here $\tilde{\xi}(y)$ characterize the perturbation profile. Since the amplitude of the perturbation is small, the corresponding solution of \eqref{eqn:general_all}, $(c(x,y;\epsilon\tilde{\xi}),p(x,y;\epsilon\tilde{\xi}))$, processes the following formal expansion with respect to $\epsilon$:
\begin{subequations}
\label{eqn:second_order_expansion}
\begin{alignat}{2}
    c(x,y;\epsilon\tilde{\xi})&=c_0(x)+\epsilon c_1(x,y;\epsilon\tilde{\xi})+O(\epsilon^2),\\
    p(x,y;\epsilon\tilde{\xi})&=p_0(x)+\epsilon p_1(x,y;\epsilon\tilde{\xi})+O(\epsilon^2),
\end{alignat}
\end{subequations}
with the zero-order terms represent the solutions to the symmetric solution solved in Subsection \ref{sec:Symmetric solutions}. When there is no ambiguity, for simplicity of notation we hide the dependence of $\epsilon\tilde{\xi}$ in the latter section. Also recall that the solutions are the combination of the inner part and outer part, i.e. $c=c^{\text{(i)}}+c^{\text{(o)}}$, $p=p^{\text{(i)}}+p^{\text{(o)}}$, and analogously for
\begin{align}
\label{eqn:inandout}
c_0=c_0^{\text{(i)}}+c_0^{\text{(o)}},\quad
c_1=c_1^{\text{(i)}}+c_1^{\text{(o)}},\quad
p_0=p_0^{\text{(i)}}+p_0^{\text{(o)}},\quad
p_1=p_1^{\text{(i)}}+p_1^{\text{(o)}}.
\end{align}
Also notice that in this set up, the outer normal vector $n$ on the free boundary $\mathcal{B}_\xi$ is given by
\begin{equation}
\label{eqn:normalvector}
    n(\epsilon\Tilde{\xi}(y),y)=\left(\frac{1}{\sqrt{1+(\epsilon\Tilde{\xi}'(y))^2}},\frac{-\epsilon\Tilde{\xi}'(y)}{\sqrt{1+(\epsilon\Tilde{\xi}'(y))^2}}\right)=\left(1+O(\epsilon^2),-\epsilon\Tilde{\xi}'(y)+O(\epsilon^3)\right).
\end{equation}
Utilizing the expansion \eqref{eqn:second_order_expansion}, expression \eqref{eqn:normalvector}, and Taylor expansion, one can evaluate $c^{\text{(i)}}$ in the following way 
\begin{subequations}
\begin{alignat}{2}
    c^{\text{(i)}}\vert_{\mathcal{B}_{\xi}}
    &=c^{\text{(i)}}(\epsilon\Tilde{\xi},y)\\\nonumber
    &=c_0^{\text{(i)}}(\epsilon\Tilde{\xi})+\epsilon c_1^{\text{(i)}}(\epsilon\Tilde{\xi},y)+O(\epsilon^2)\\\nonumber
    &=c_0^{\text{(i)}}(0)+\epsilon\Tilde{\xi}\frac{\partial}{\partial x}c_0^{\text{(i)}}(0)+\epsilon  c_1^{\text{(i)}}(0,y)+O(\epsilon^2),
\end{alignat}
and the normal derivative $\frac{\partial}{\partial n}c^{\text{(i)}}$ is given by 
\begin{align}
   \label{eqn:normalderv}
    \frac{\partial}{\partial n}c^{\text{(i)}}\vert_{\mathcal{B}_{\xi}}
    &=\left(\frac{\partial}{\partial x}c^{\text{(i)}}(\epsilon\Tilde{\xi},y),\frac{\partial}{\partial y}c^{\text{(i)}}(\epsilon\Tilde{\xi},y)\right)\cdot n(\epsilon\Tilde{\xi}(y),y)\\\nonumber
    &=\left(\frac{\partial}{\partial x}c_0^{\text{(i)}}(\epsilon\Tilde{\xi})+\epsilon \frac{\partial}{\partial x}c_1^{\text{(i)}}(\epsilon\Tilde{\xi},y)+O(\epsilon^2),\epsilon \frac{\partial}{\partial y}c_1^{\text{(i)}}(\epsilon\Tilde{\xi},y)+O(\epsilon^2)\right)\cdot n(\epsilon\Tilde{\xi}(y),y)\\\nonumber
    &=\frac{\partial}{\partial x}c_0^{\text{(i)}}(0)+\epsilon\Tilde{\xi}\frac{\partial^2}{\partial x^2}c_0^{\text{(i)}}(0)+\epsilon \frac{\partial}{\partial x}c_1^{\text{(i)}}(0,y)+O(\epsilon^2).    
\end{align}
Similarly, for $ c^{\text{(o)}}$, $p^{\text{(i)}}$, and $p^{\text{(o)}}$ we have 
\begin{alignat}{2}
c^{\text{(o)}}\vert_{\mathcal{B}_{\xi}}
&=c_0^{\text{(o)}}(0)+\epsilon\Tilde{\xi}\frac{\partial}{\partial x}c_0^{\text{(o)}}(0)+\epsilon c_1^{\text{(o)}}(0,y)+O(\epsilon^2),\\
\frac{\partial}{\partial n}c^{\text{(o)}}\vert_{\mathcal{B}_{\xi}}
&=\frac{\partial}{\partial x}c_0^{\text{(o)}}(0)+\epsilon\Tilde{\xi}\frac{\partial^2}{\partial x^2}c_0^{\text{(o)}}(0)+\epsilon \frac{\partial}{\partial x}c_1^{\text{(o)}}(0,y)+O(\epsilon^2),\\
p^{\text{(i)}}\vert_{\mathcal{B}_{\xi}}
&= p^{\text{(i)}}_0(0)+\epsilon\Tilde{\xi}\frac{\partial}{\partial x}p^{\text{(i)}}_0(0)+\epsilon p^{\text{(i)}}_1(0,y)+O(\epsilon^2),\\
\frac{\partial}{\partial n}p^{\text{(i)}}\vert_{\mathcal{B}_{\xi}}&=\frac{\partial}{\partial x}p^{\text{(i)}}_0(0)+\epsilon\Tilde{\xi}\frac{\partial^2}{\partial x^2}p^{\text{(i)}}_0(0)+\epsilon \frac{\partial}{\partial x}p^{\text{(i)}}_1(0,y)+O(\epsilon^2),\\
p^{\text{(o)}}\vert_{\mathcal{B}_{\xi}}
&= p^{\text{(o)}}_0(0)+\epsilon\Tilde{\xi}\frac{\partial}{\partial x}p^{\text{(o)}}_0(0)+\epsilon p^{\text{(o)}}_1(0,y)+O(\epsilon^2),\\
\frac{\partial}{\partial n}p^{\text{(o)}}\vert_{\mathcal{B}_{\xi}}
&=\frac{\partial}{\partial x}p^{\text{(o)}}_0(0)+\epsilon\Tilde{\xi}\frac{\partial^2}{\partial x^2}p^{\text{(o)}}_0(0)+\epsilon \frac{\partial}{\partial x}p^{\text{(o)}}_1(0,y)+O(\epsilon^2).
\end{alignat}
\end{subequations}
Plugging the expansion \eqref{eqn:second_order_expansion} into \eqref{eqn:general_all}, the zero-order terms are canceled out, and we collect the terms of order $O(\epsilon)$. Regarding the nutrient, the first order terms solve the following boundary value problem
\begin{subequations}
\label{eqn: invivo_general_linearpart_n}
\begin{alignat}{2}
-\Del c_1^{\text{(i)}} + \lam c_1^{\text{(i)}} &= 0,\\
-\Del c_1^{\text{(o)}} + c_1^{\text{(o)}} &= 0,\\
c_1^{\text{(i)}}(0,y)&=c_1^{\text{(o)}}(0,y),\\
\Tilde{\xi}\cdot\frac{\partial^2}{\partial x^2}c_0^{\text{(i)}}(0)+\frac{\partial}{\partial x}c_1^{\text{(i)}}(0,y)
&=\Tilde{\xi}\cdot\frac{\partial^2}{\partial x^2}c_0^{\text{(o)}}(0)+\frac{\partial}{\partial x}c_1^{\text{(o)}}(0,y),\\
c_1^{\text{(i)}}(-\infty,y)&<\infty,\\
c_1^{\text{(o)}}(+\infty,y)&<\infty.
\end{alignat}
\end{subequations}
While, for the pressure, the first order terms solve
\begin{subequations}
\label{eqn: invivo_general_linearpart_p}
\begin{alignat}{2}
-\Del p^{\text{(i)}}_1 &= G_0 c_1^{\text{(i)}},\\
-\Del p^{\text{(o)}}_1 &= G_0 c_1^{\text{(o)}},\\
\Tilde{\xi}\cdot\frac{\partial}{\partial x}p^{\text{(i)}}_0(0)+p^{\text{(i)}}_1(0,y) &=\Tilde{\xi}\cdot\frac{\partial}{\partial x}p^{\text{(o)}}_0(0)+p^{\text{(o)}}_1(0,y)= 0,\\
\Tilde{\xi}\cdot\frac{\partial^2}{\partial x^2}p_0^{\text{(i)}}(0)+\frac{\partial}{\partial x}p_1^{\text{(i)}}(0,y)
&=\Tilde{\xi}\cdot\frac{\partial^2}{\partial x^2}p_0^{\text{(o)}}(0)+\frac{\partial}{\partial x}p_1^{\text{(o)}}(0,y),\\
p^{\text{(i)}}_1(-\infty,y)&<\infty.
\end{alignat}
\end{subequations}
By now, we finished deriving the equations and boundary conditions for the first-order terms. 

In the next section, we further show that when the boundary profile $\Tilde{\xi}(y)$ is close to $\cos{ly}$, the first-order terms can be solved explicitly, from which one can figure out the so-called boundary evolution function.

\subsection{Single mode perturbation}
\label{sec:The linearized perturbed solutions}
We further investigate the situation when the perturbed profile is in the form of $\Tilde{\xi}(y)=\cos{ly}+O(\epsilon)$. In this context, the first-order terms presented in the last subsection, $(c^{\text{(i)}}_1(x,y), p^{\text{(i)}}_1(x,y), c^{\text{(o)}}_1(x,y), p^{\text{(o)}}_1(x,y))$, can be solved explicitly, and from this, we further determine the boundary evolution equation (see \eqref{eqn:E} for specific formula). This evolution function characterizes the boundary evolution tendency, and the properties of it play a critical role in determining the bifurcation point (see equation \eqref{eqn:bifurcation point}) and verifying the conditions in Crandrall-Rabinowitz Theorems in the later sections. Thus, we investigate it at the end of this subsection (see Proposition \ref{rem:F1}).

The first-order terms capture the main reaction to the perturbation. When $\xi(y)=\epsilon\cos{ly}+O(\epsilon^2)$, the first order terms in \eqref{eqn:second_order_expansion} can be separated in the following way (see Section 3.2 in \cite{feng2022tumor} for a more rigorous justification):
\begin{subequations}
\label{eqn:separation_linear_invivo}
\begin{alignat}{2}
    c_1^{\text{(i)}}(x,y;\epsilon\Tilde{\xi})&=\Tilde{c}_{1,l}^{\text{(i)}}(x)\cos{ly},\qquad c_1^{\text{(o)}}(x,y;\epsilon\Tilde{\xi})&=\Tilde{c}_{1,l}^{\text{(o)}}(x)\cos{ly},\\
    p^{\text{(i)}}_1(x,y;\epsilon\Tilde{\xi})&=\Tilde{p}_{1,l}^{\text{(i)}}(x)\cos{ly},\qquad p^{\text{(o)}}_1(x,y;\epsilon\Tilde{\xi})&=\Tilde{p}_{1,l}^{\text{(o)}}(x)\cos{ly}.
\end{alignat}
\end{subequations}
By plugging  \eqref{eqn:separation_linear_invivo} into \eqref{eqn: invivo_general_linearpart_n} and \eqref{eqn: invivo_general_linearpart_p}, the equations reduce to
\begin{subequations}
\begin{alignat}{2}
-\partial_x^2 \Tilde{c}_{1,l}^{\text{(i)}}+(\lambda+l^2)\Tilde{c}_{1,l}^{\text{(i)}}&=0,&\quad \text{for}\quad x\leq 0,\\
-\partial_x^2 \Tilde{c}_{1,l}^{\text{(o)}}+(1+l^2)\Tilde{c}_{1,l}^{\text{(o)}}&=0,&\quad \text{for}\quad x\geq 0,\\
-\partial_x^2 \Tilde{p}_{1,l}^{\text{(i)}}+l^2\Tilde{p}_{1,l}^{\text{(i)}} &= G_0  \Tilde{c}_{1,l}^{\text{(i)}},&\quad \text{for}\quad x\leq 0,\\
-\partial_x^2 \Tilde{p}_{1,l}^{\text{(o)}}+l^2\Tilde{p}_{1,l}^{\text{(o)}} &= G_0  \Tilde{c}_{1,l}^{\text{(o)}},&\quad \text{for}\quad x\geq 0,
\end{alignat}
and the boundary conditions reduce to
\begin{alignat}{2}
\partial_x c_0^{\text{(i)}}(0)+\Tilde{c}_{1,l}^{\text{(i)}}(0)&=
\partial_x c_0^{\text{(o)}}(0)+\Tilde{c}_{1,l}^{\text{(o)}}(0),\\
\partial_x^2 c_0^{\text{(i)}}(0)+\partial_x\Tilde{c}_{1,l}^{\text{(i)}}(0)&=
\partial_x^2 c_0^{\text{(o)}}(0)+\partial_x\Tilde{c}_{1,l}^{\text{(o)}}(0),\\
\partial_x p_0^{\text{(i)}}(0)+\Tilde{p}_{1,l}^{\text{(i)}}(0)&=\partial_x p_0^{\text{(o)}}(0)+\Tilde{p}_{1,l}^{\text{(o)}}(0)=0,\\
\partial_x^2 p_0^{\text{(i)}}(0)+\partial_x\Tilde{p}_{1,l}^{\text{(i)}}(0)&=
\partial_x^2 p_0^{\text{(o)}}(0)+\partial_x\Tilde{p}_{1,l}^{\text{(o)}}(0),
\end{alignat}
and with,
\begin{equation}
c_1^{\text{(i)}}(-\infty)<\infty,\quad c_1^{\text{(o)}}(+\infty)<\infty,\quad p_1^{\text{(i)}}(-\infty)<\infty,
\end{equation}
\end{subequations}
recall that we do not require $ p_1^{\text{(o)}}(+\infty)<\infty$ as mentioned before. Finally, by solving the above boundary value problems, we get:
\begin{subequations}
\label{eqn: linear_soln_invivo}
\begin{alignat}{2}
\label{eqn:invivo_c1i}
\Tilde{c}_{1,l}^{\text{(i)}}(x)&=-\frac{\sqrt{\lam}\cdot c_B}{\sqrt{\lam+l^2}+\sqrt{1+l^2}}e^{\sqrt{\lam+l^2}x},\\
\Tilde{c}_{1,l}^{\text{(o)}}(x)&=-\frac{\sqrt{\lam}\cdot c_B}{\sqrt{\lam+l^2}+\sqrt{1+l^2}}e^{-\sqrt{1+l^2}x},\\
\Tilde{p}_{1,l}^{\text{(i)}}(x)&=\frac{G_0c_B}{\sqrt{\lam}}\left(\left(\frac{1}{\sqrt{\lambda}+1}-\frac{1}{\sqrt{\lambda+l^2}+\sqrt{1+l^2}}\right)e^{l x}+\frac{e^{\sqrt{\lam+l^2}x}}{\sqrt{\lam+l^2}+\sqrt{1+l^2}}\right),\\
\Tilde{p}_{1,l}^{\text{(o)}}(x)&=-G_0\Tilde{c}_{1,l}^{\text{(o)}}(x)+C_1e^{lx}+C_2e^{-lx},
\label{eqn:p_1(x)}
\end{alignat}
\end{subequations}
where $C_1$ and $C_2$ are some constants that can be solved explicitly and do not contribute to the proofs later, so we omit their expressions here.

Observe that,  when $\xi(y)=\epsilon\Tilde{\xi}(y)=\epsilon\cos{ly}+O(\epsilon^2)$, the normal speed on the boundary, $v(\xi(y),y)$, is governed by Darcy's law as follows:
\begin{align}
\label{eqn:single_speed}
v(\xi(y),y)
&=-\frac{\partial p^{\text{(i)}}}{\partial n}(\xi(y),y)\\\nonumber
&=-\partial_x p^{\text{(i)}}_0(0)-\left(\partial_x^2 p_0^{\text{(i)}}(0)+\partial_x\Tilde{p}_{1,l}^{\text{(i)}}(0)\right)\epsilon\cos{ly}+O(\epsilon^2)\\\nonumber
&=v_0(\lam)+E(\lam,l)\epsilon\cos{ly}+O(\epsilon^2),
\end{align}
where $v_0(\lam)$ was defined in \eqref{eqn: invivo_v0}, and the expression of $E(\lam,l)$ is given by
\begin{align}
\label{eqn:E}
E(\lambda,l):=\frac{G_0 c_B}{\sqrt{\lambda}}\left(\frac{\sqrt{\lambda}-l}{\sqrt{\lambda}+1}+\frac{l-\sqrt{\lambda+l^2}}{\sqrt{\lambda+l^2}+\sqrt{1+l^2}}\right).
\end{align}
Intuitively, the zero-order term \( v_0(\lambda) \) represents the propagation speed of the wave, while \( E(\lambda, l) \) is the so-called boundary evolution equation, which describes the evolution of the perturbation amplitude. If \( E(\lambda, l) > 0 \), the amplitude tends to grow, and if it takes negative values, the amplitude will decay, as discussed in Section 4 of \cite{feng2022tumor}. Therefore, for a suitable \( \lambda_0^l \) (which will later serve as the bifurcation point) such that \( E(\lambda_0^l, l) = 0 \), the perturbation amplitude is stable up to \( O(\epsilon) \). It is thus reasonable to expect that nonsymmetric solutions can be found near \( \xi(y) = \epsilon \cos{ly} \) and \( \lambda = \lambda_0^l \) by performing higher-order corrections. To substantiate the existence of such \( \lambda_0^l \) and verify some useful properties for the Crandall-Rabinowitz theorem later, we plot \( E(\lambda, l) \) in Figure 1. Based on the graphs, we conclude that $E(\lam,l)$ exhibits the following properties.
\begin{figure}
\label{fig:F1}
  \begin{center}
    \includegraphics[width=0.49\textwidth]{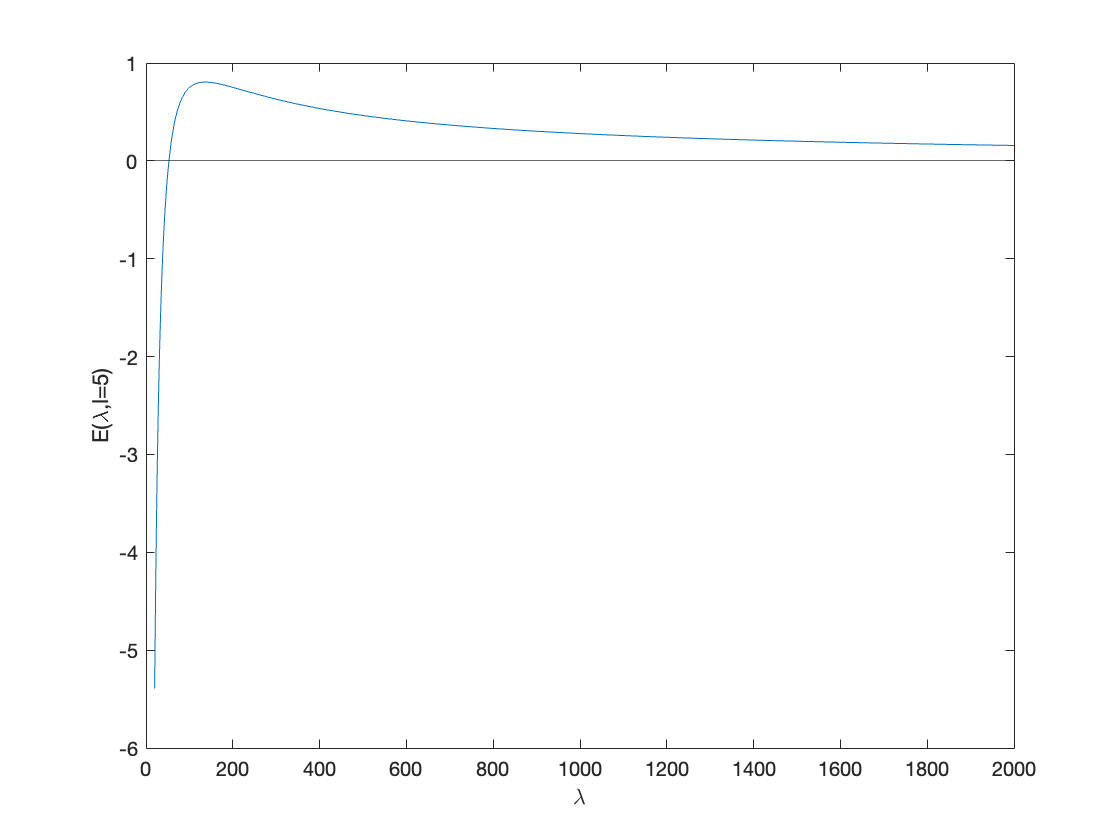}
    \includegraphics[width=0.49\textwidth]{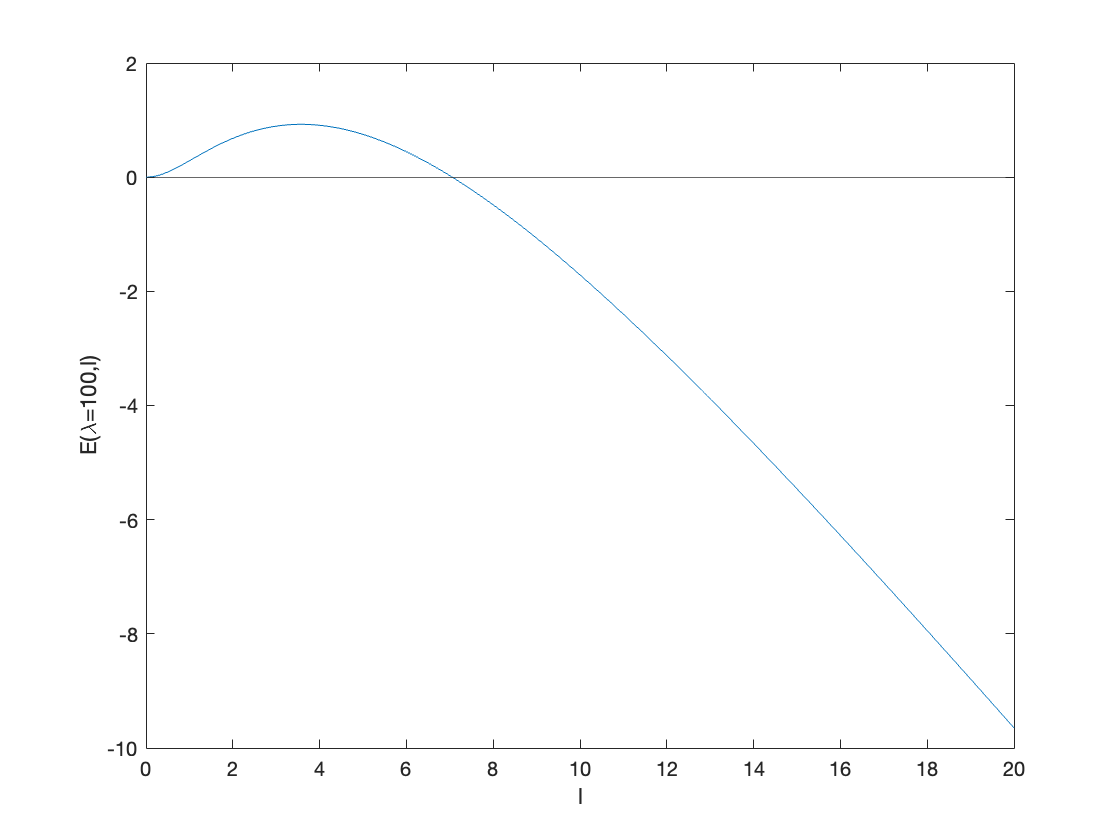}
    \caption{Fix $G_0\cdot c_B=100$. Left: plot for $E(\lam,l)$ with $l=5$ and $\lam\in[0,2000]$. Right: plot for $E(\lam,l)$ with $\lam=100$ and $l\in[0,20]$.}
  \end{center}
\end{figure}

\begin{prop}
\label{rem:F1}
 (1) Given any integer $l>0$, there exists unique $\lam_0^l>l^2$ such that 
    \begin{equation}
    \label{eqn:bifurcation point}
    E(\lam_0^l,l)=\frac{G_0 c_B}{\sqrt{\lam_0^l}}\left(\frac{\sqrt{\lam_0^l}-l}{\sqrt{\lam_0^l}+1}+\frac{l-\sqrt{\lam_0^l+l^2}}{\sqrt{\lam_0^l+l^2}+\sqrt{1+l^2}}\right)=0.
    \end{equation}

   (2) Given any $\lam>1$, there exists a unique $l_0>0$, which may not be an integer, such that $E(\lam, l_0)=0$. Furthermore, $E(\lam,l)\neq 0$ for any $l\neq 0$ or $l\neq l_0$. 

\end{prop}

    The fact $\lam_0^l>l^2$ holds, since one can easily check that 
    \begin{equation*}
       E(\lam,l)< \frac{G_0 c_B}{\sqrt{\lam}}\left(\frac{l-\sqrt{\lam+l^2}}{\sqrt{\lam+l^2}+\sqrt{1+l^2}}\right)<0,\quad\text{for any \,$0<\lam\leq l^2$.}
    \end{equation*}
   
The existence of $\lam_0^l$ can be shown by checking the limits and applying the intermediate value theorem
   \begin{align*}
       \lim_{\lam\rightarrow 0}E(\lam,l)&=-\infty,\\
       \lim_{\lam\rightarrow +\infty}E(\lam,l)&=\lim_{x\rightarrow0^+}E(\tfrac{1}{x^2},l)=\lim_{x\rightarrow 0^+}G_0 c_B(\sqrt{1+l^2}-1)x^2+o(x^2)=0^+.
   \end{align*}
    And, the existence of $l_0$ was proved similarly in \cite{feng2022tumor} (see Corollary 3) by checking the asymptote. However, to show uniqueness of $\lam
    _0^l$ and $l_0$ one needs to study the behavior of the derivatives, which are given by
\begin{align}
\label{eqn:partiallamE}
\partial_{\lam}E(\lam,l)
&=-\frac{G_0 c_B}{2\lambda^{3/2}}\left(\frac{\sqrt{\lambda}-l}{\sqrt{\lambda}+1}+\frac{l-\sqrt{\lambda+l^2}}{\sqrt{\lambda+l^2}+\sqrt{1+l^2}}\right)\\\nonumber
&\quad+\frac{G_0 c_B}{2\sqrt{\lam}}\left(\frac{l+1}{\sqrt{\lam}(\sqrt{\lam}+1)^2}-\frac{l+\sqrt{1+l^2}}{\sqrt{\lam+l^2}(\sqrt{\lam+l^2}+\sqrt{1+l^2})^2}\right);\\
\partial_{l}E(\lam,l)
&=\frac{G_0 c_B}{\sqrt{\lam}}\left(\frac{-1}{\sqrt{\lam}+1}+\frac{(\sqrt{1+l^2}+l)(\sqrt{\lam+l^2}-l)}{(\sqrt{1+l^2}+\sqrt{\lam+l^2})\sqrt{\lam+l^2}\sqrt{1+l^2}}\right).
\end{align}
Rigorous verification of the sign of the above expression is difficult, but elementary, and does not deepen the understanding of the main content. In addition, it can be easily verified by plotting the curves using the above explicit expressions. Therefore, we decided to omit this part of the proof.

\subsection{A nonlinear functional and its Fr\'{e}chet derivative.}
We introduce a nonlinear functional map inspired by \eqref{eqn:single_speed}, reducing the problem of proving the existence of nonsymmetric traveling wave solutions for \eqref{eqn:general_all} to verifying that the kernel of this map contains nontrivial bifurcation branch. To achieve this, we compute its Fréchet derivative and demonstrate that it can be formulated as an eigenvalue problem, where the eigenvalues coincide with $E(\lam,l)$; see \eqref{eqn: eigenprob}. This facilitates the verification of the assumptions in the Crandall-Rabinowitz theorem later.

To begin with, we introduce the functional spaces for the boundary profile $\xi(y)$:
\begin{subequations}
\label{def: two_spaces}
\begin{alignat}{2}
    X^{m+\alpha}&=\left\{f(y)\in C^{m+\alpha}: f(y)\text{ is }2\pi\text{ periodic and even}\right\},\\
    X_1^{m+\alpha}&=\text{closure of the linear space spanned by } \{\cos{jy}, j=0,1,2,\dots\} \text{ in } X^{m+\alpha}.
\end{alignat}
\end{subequations}
Note that all modes are included. Thus, any $\xi(y)\in X_1^{m+\alpha}$ can be represented as a Fourier series. 

Inspired by the calculations in \eqref{eqn:single_speed}, we introduce the following nonlinear functional map: 
\begin{equation}
\label{eqn: bifurcation_fcnal}
F: X_1^{3+\alpha}\times \mathbb{R} \to X_1^{2+\alpha},\quad  (\xi,\lambda) \mapsto F(\xi,\lambda):=-\frac{\partial p}{\partial x}(\xi(y),y;\xi)-v_0(\lam).    
\end{equation}
Regarding the expression of $F(\xi,\lam)$, $p(x,y;\xi)$ stands for the pressure function in \eqref{eqn:general_all} associated with the boundary profile $\xi(y)\in X_1^{3+\alpha}$ and the nutrient consumption rate $\lam>0$; and $v_0(\lam)$, defined in \eqref{eqn: invivo_v0}, represents the travel speed of the solution. 

By a computation similar to that in \eqref{eqn:normalderv}, we observe that the normal derivative of $p$, which determines the boundary moving speed, is approximately $-\frac{\partial p}{\partial x}$, with an error of order $O(\epsilon^2)$. Consequently, when the perturbation amplitude is small, $F$ quantifies the difference between the boundary moving speed of the perturbed problem and the symmetric traveling wave speed associated with $\lambda$. Following a similar argument as in \cite{borisovich2005symmetry} (see Equation (3.20)) and \cite{lu2022bifurcation} (see Equation (27)), we assert that $(\xi,\lambda)$ forms a traveling wave solution to \eqref{eqn:general_all} if and only if
\begin{equation} 
\label{eqn: F=0} 
F(\xi,\lambda)=0. 
\end{equation}
The symmetric solutions naturally correspond to the trivial solution $(0,\lambda)$, while nonsymmetric solutions $(\xi_\epsilon,\lambda_\epsilon)$, as defined in \eqref{eqn:nontriviasoln}, correspond to nontrivial solutions of \eqref{eqn: F=0}.

Next, we determine the Fr\'{e}chet derivative of $F(\xi,\lam)$ (with respect to $\xi$) at $(0,\lam)$, denoted as $F_{\xi}(0,\lam)$. To do this, one needs to rigorously justify the expansion in \eqref{eqn:second_order_expansion}, which is given by the following two lemmas.
\begin{lem}
\label{lem:justification_I}
For any nutrient consumption rate $\lam>0$, perturbation amplitude $0<\epsilon\ll 1$, and perturbation profile $\Tilde{\xi}(y)\in C^{3+\alpha}(\mathbb{R})$, let $(c,p)$ be the solution of the corresponding perturbation problem defined in \eqref{eqn:second_order_expansion}, then
\begin{align*}
    \left\Vert c(x,y;\epsilon\Tilde{\xi})-c_0(x)\right\Vert_{C^{1+\alpha}(\Omega)}&\leq C\vert\epsilon\vert\Vert\Tilde{\xi}\Vert_{C^{3+\alpha}(\mathbb{R})},\\
    \left\Vert p(x,y;\epsilon\Tilde{\xi})-p_0(x)\right\Vert_{C^{3+\alpha}(\Omega)}&\leq C\vert\epsilon\vert\Vert\Tilde{\xi}\Vert_{C^{3+\alpha}(\mathbb{R})},
\end{align*}
where $C$ is a constant independent of $\epsilon$, and $(c_0, p_0)$ stands for the unperturbed solutions given in \eqref{eqn: invivo_unperturbed_solns}.
\end{lem}
\begin{lem}
\label{lem:justification_II}
For any nutrient consumption rate $\lam>0$, perturbation amplitude $0<\epsilon\ll 1$, and perturbation profile $\Tilde{\xi}(y)\in C^{3+\alpha}(\mathbb{R})$, let $(c,p)$ be the solution of the corresponding perturbation problem defined in \eqref{eqn:second_order_expansion}, then
\begin{align*}
    \Vert c(x,y;\epsilon\Tilde{\xi})-c_0(x)-\epsilon \hat{c}_1(x,y;\epsilon\Tilde{\xi})\Vert_{C^{1+\alpha}(\Omega)}&\leq C\vert\epsilon\vert^2\Vert\Tilde{\xi}\Vert_{C^{3+\alpha}(\mathbb{R})},\\
    \Vert p(x,y;\epsilon\Tilde{\xi})-p_0(x)-\epsilon \hat{p}_1(x,y;\epsilon\Tilde{\xi})\Vert_{C^{3+\alpha}(\Omega)}&\leq C\vert\epsilon\vert^2\Vert\Tilde{\xi}\Vert_{C^{3+\alpha}(\mathbb{R})},
\end{align*}
where $C$ is a constant independent of $\epsilon$, and $(c_0, p_0)$ stands for the unperturbed solutions, $(\hat{c}_1, \hat{p}_1)$ corresponds to the Hanzawa transformation (see \eqref{eqn: HZ transform} and \eqref{eqn:twohat}) of the first-order terms.
\end{lem}
The proof of the above lemmas is standard but cumbersome. For the sake of exposition, we provide
a sketch of the proof in Appendix \ref{Justification of the expansion}. Using the above lemmas, one can determine the Fr\'{e}chet derivative $F_\xi(0,\lam)$, we summarize it in the following proposition.
\begin{prop}
\label{lem:Fder}
With the same assumptions as in Lemma \ref{lem:justification_I} and Lemma \ref{lem:justification_II}, and $\xi(y)=\epsilon\Tilde{\xi}(y)$. Then, the Fr\'{e}chet derivative $F_\xi(0,\lam)$ is given by 
\begin{equation}
\label{eqn:Fderv0}
    \left[F_{\xi}(0,\lam)\right]\Tilde{\xi}=-\Tilde{\xi}\frac{\partial^2 p_0}{\partial x^2}(0)-\frac{\partial \hat{p}_1}{\partial x}(0,y;\epsilon\Tilde{\xi}),
\end{equation}
where $p_0$ is given in \eqref{eqn: invivo_unperturbed_solns}, and $\hat{p}_1(x,y;\epsilon\Tilde{\xi})$ is defined in \eqref{eqn: HZ transform}. If $\Tilde{\xi}(y)$ further belongs to $X_1^{3+\alpha}$, i.e.,  $\Tilde{\xi}(y)$ can be represented as a Fourier series, denoted as $\Tilde{\xi}(y)=\sum_{l=1}^{\infty} a_l\cos{ly}$. Then,  \eqref{eqn:Fderv0} reduce to
\begin{equation}
\label{eqn:Fderv1}
    \left[F_{\xi}(0,\lam)\right]\Tilde{\xi}
    =\sum_{l=1}^{\infty} a_l E(\lam,l)\cos{ly},
\end{equation}
where the eigenvalue $E(\lam,l)\in\mathbb{R}$ is given by \eqref{eqn:E}. 
\end{prop}
\begin{rem}
\label{rem:linearity}
The operator defined in \eqref{eqn:Fderv0} is linear. In particular, the second term on the right-hand side of \eqref{eqn:Fderv0} shall be viewed as a nonlocal linear operator that maps the shape of the boundary $\xi(y)=\epsilon\Tilde{\xi}(y)$ to $-\frac{\partial \hat{p}_1}{\partial x}(0,y;\epsilon\Tilde{\xi})$. And the linearity can be seen more clearly when $\Tilde{\xi}$ belongs further to $X_1^{3+\alpha}$, as shown in \eqref{eqn:Fderv1}.
\end{rem}
\label{rem:eproblem}
\begin{rem}
When $\Tilde{\xi}(y)\in X_1^{3+\alpha}$, the Fr\'{e}chet derivative $F_\xi(0,\lam)$ is characterized by the eigenvalue problem:
\begin{equation}
\label{eqn: eigenprob}
\left[F_{\xi}(0,\lam)\right]\cos{ly}=E(\lam,l)\cos{ly}.
\end{equation}
\end{rem}
In the following, we provide a proof of Proposition \ref{lem:Fder}.
\begin{proof}
By using Lemma \ref{lem:justification_I} and Lemma \ref{lem:justification_II}, and the fact that $F(0,\lam)=0$, one has
\begin{equation}
\label{eqn:Festima1}
  \left\Vert F(\epsilon\Tilde{\xi},\lambda)-F(0,\lambda)-\epsilon\left(-\Tilde{\xi}\partial_x^2 p_0(0)-\partial_x \hat{p}_1(0,y;\epsilon\Tilde{\xi})\right)\right\Vert_{C^{2+\alpha}(\mathbb{R})}\leq C\epsilon^2.  
\end{equation}
Moreover, one can easily verify that
\begin{align}
\label{eqn:Festima2}
    \sup_{\|\Tilde{\xi}\|_{C^{3+\alpha}(\mathbb{R})}\leq 1}\left\Vert\Tilde{\xi}\partial_x^2 p_0(0)+\partial_x \hat{p}_1(0,y;\epsilon\Tilde{\xi})\right\Vert_{C^{2+\alpha}(\mathbb{R})}\leq C.
\end{align}
In fact, since $p_0(x)$ has an explicit formula, the boundedness of the first term follows directly. While, the boundedness of the second follows a similar proof of Lemma \ref{lem:justification_II} (apply classical elliptic estimates to equation \eqref{eqn: invivo_general_linearpart_p}). Then, estimates \eqref{eqn:Festima1} and \eqref{eqn:Festima2} together yield the map defined in \eqref{eqn:Fderv0} is a bounded linear operator and therefore it is indeed a Fr\'{e}chet derivative.

For the second part, when $\xi\in X_1^{3+\alpha}$ and takes the form $\xi(y)=\sum_{l}a_l\cos{ly}$. By using the orthogonality of the Fourier basis, one can solve each Fourier mode in the first-order term separately in the same manner as in Section \ref{sec:The linearized perturbed solutions}. Then, \eqref{eqn:Fderv0} further reduce to 
\begin{align}
\label{eqn:decompostion}
    \left[F_{\xi}(0,\lam)\right]\xi
    &=\sum_l a_l\left(-\partial_x^2 p_0(0)-\partial_x \tilde{p}_{1,l}(0)\right)\cos{ly},\\\nonumber
    &=\sum_l a_l E(\lam,l)\cos{ly}.
\end{align}
By now, we have completed the proof.
\end{proof}
In the following subsection, we show that based on the eigenvalue problem \eqref{eqn: eigenprob} and the properties of $E(\lam,l)$ established in Proposition \ref{rem:F1}, we can determine bifurcation points and further conclude the existence of desired bifurcation branches in equation \eqref{eqn: F=0}.
\subsection{Existence of nontrivial bifurcation branches} 
\label{sec:Existence of nontrivial bifurcation branches}
In \cite{borisovich2005symmetry}, Borisovich and Friedman applied the Crandall-Rabinowitz theorem to establish the existence of nonradially symmetric solutions to a tumor growth model derived from \cite{greenspan1976growth}. As discussed in the Introduction, the model described by \eqref{eqn:general_all} is derived from the incompressible limit of the PME, which makes it fundamentally different from the models developed from \cite{greenspan1976growth}. In the latter models, non-radial symmetric solutions are typically sought. In contrast, in the PME-derived models, the focus shall shift to finding nonsymmetric traveling wave solutions. Fortunately, we are able to show that the bifurcation analysis framework introduced by Friedman, using the Crandall-Rabinowitz theorem, is still applicable. We carry out the proof in this subsection. For the reader’s convenience, we first recall the Crandall-Rabinowitz theorem below.

\begin{thm}
\label{thm:Crandall}
    Let $W,Z$ be real Banach spaces and $\calF(w,\mu)$ be a $C^{p}$ map, $p\geq 3$, which maps a neighborhood of $(0,\mu_0)$ in $W\times\mathbb{R}$ into $Z$. For any $\nu\in\mathbb{R}$, the Fr\'{e}chet derivative $\calF_w(0,\nu):=\frac{d}{dw}F(0,\nu)$ maps $W$ to $Z$. Suppose
    \begin{enumerate}
        \item $\calF(0,\mu)=0$ for all $\mu$ in a neighborhood of $\mu_0$,
        \item The kernel space of $\calF_w(0,\nu)$ is of one dimensional spanned by $w_0\in W$, i.e.,  $\text{Ker}\{\calF_w(0,\mu_0)\}=\text{span}\{w_0\}$.
        \item The range of $\calF_w(0,\mu_0)$ has codimension $1$, that is, $\text{dim}\{Z/Z_1\}=1$ with $Z_1=\text{Img}\{\calF_w(0,\mu_0)\}\subseteq Z$.
        \item The mix derivative $\calF_{w\mu}(w,r)$, $w\in W, r\in\mathbb{R}$, satisfies $\left[\calF_{w\mu}(0,\mu_0)\right]w_0\notin Z_1$.
    \end{enumerate}
    Then, $(0,\mu_0)$ is a bifurcation point of equation $\calF(w,\mu)=0$ in the following sense: In a neighborhood of $(0,\mu_0)$, the set of solutions of $\calF(w,\mu)=0$ consists of two $C^{p-2}$ smooth curves $\mathcal{C}_1$ and $\mathcal{C}_2$, which intersect only at the point $(0,\mu_0)$. Moreover, $\mathcal{C}_1$ is the curve $(0,\mu)$ and $\mathcal{C}_2$ can be parameterized by a small parameter $\epsilon$ as follows:
    \begin{equation*}
        \mathcal{C}_2: (w(\epsilon),\mu(\epsilon)),\, \epsilon \text{ small},\, (w(0),\mu(0))=(0,\mu_0),\, w'(0)=w_0.
    \end{equation*}
\end{thm}
To apply Theorem \ref{thm:Crandall} to the nonlinear map \eqref{eqn: bifurcation_fcnal}, the key step is to identify a bifurcation point \((0, \lambda_0)\) such that the Fréchet derivative at this point, \(F_{\xi}(0, \lambda_0)\), satisfies the conditions outlined in Theorem \ref{thm:Crandall}. We demonstrate that this can be achieved using the properties of $E(\lam,l)$ established in Proposition \ref{rem:F1}. Consequently, \eqref{eqn:general_all} admits nonsymmetric traveling wave solutions. This main result is summarized in the following main theorem.

\begin{thm}
\label{thm: mainthm}
Consider the nonlinear map \eqref{eqn: bifurcation_fcnal}, which maps $ X_1^{3+\alpha}\times\mathbb{R}$ to $X_1^{2+\alpha}$. Assume $0<\epsilon\ll 1$. Then for each positive integer $l$, there exists a $\lam_0^l>l^2$ such that $(0,\lam_0^l)$ is a bifurcation point to $F(\xi,\lam)=0$ in the following sense: In a neighborhood of $(0,\lam_0^l)$, the set of solutions of $F(\xi,\lam)=0$ consists of two smooth curves $\mathcal{C}_1$ and $\mathcal{C}_2$ that intersect only at the point $(0,\lam_0^l)$. Moreover, $\mathcal{C}_1$ is the curve $(0,\lam)$ and $\mathcal{C}_2$ can be parameterized as follows:
    \begin{equation*}
        \mathcal{C}_2: (\xi_\epsilon,\lam_\epsilon):=(\xi(\epsilon),\lam(\epsilon)),\, \text{with $\epsilon$ small},\,(\xi(0),\lam(0))=(0,\lam_0^l),\, \xi'(0)=\cos{ly}.
    \end{equation*}
\end{thm}
\begin{rem}
As we interpreted before, the nontrivial branch ($ \mathcal{C}_2$) corresponds to nonsymmetric traveling wave solutions to \eqref{eqn:general_all} with the steady boundary profile $\xi_\epsilon(y)$ and consumption rate $\lam_\epsilon$. Moreover, the traveling speed is given by $v_0(\lam_\epsilon)$ (recall \eqref{eqn: invivo_v0}).
\end{rem}
\begin{rem}
Note that by applying Crandrall-Rabinowitz theorem, we do not obtain any information of $\lam'(0)$. In fact, it can be computed using Corollary 2.3 in \cite{liu2007imperfect}, this has been done in a recent paper \cite{zhao2025determination} for a different Hele-Shaw type tumor growth model. However, to calculate $\lam'(0)$, one needs to expand the perturbed solution to the order of $O(\epsilon^2)$, which brings much more computation. We do not provide such a calculation in this paper, but we point out that similar to the calculation in \cite{zhao2025determination}, as a two-dimensional model, by using the symmetric of the basis function, one can verify similarly that $\lam'(0)=0$.
\end{rem}

Finally, we provide the proof of Theorem \ref{thm: mainthm}.
\begin{proof}
According to Proposition \ref{rem:F1}, given any positive integer $l$, we can find a unique $\lam_0^l>l^2$ such that $E(\lam_0^l,l)=0$. Then we show that $(0,\lam_0^l)$ is indeed a bifurcation point to \eqref{eqn: F=0} by verifying that the map $F(\xi,\lam)$ indeed satisfies the conditions to apply Theorem \ref{thm:Crandall} with the setting $W=X_1^{3+\alpha}$, $Z=X_1^{2+\alpha}$, $w=\xi$, $w_0=\cos{ly}$, $\mu=\lambda$, and $\mu_0=\lambda_0^l$.

For the differentiability of $F$, it is equivalent to establishing the regularity of the corresponding PDEs. Firstly, note that the structure of the PDEs guarantees that $F$ maps $X_1^{3+\alpha}\times\mathbb{R}$ to $X_1^{2+\alpha}$. Secondly, using classical elliptic estimates and the Sobolev imbedding theory, one can justify that $F(\xi,\mu)$ is differentiable to any order by repeating the process in the same way as Lemma \ref{lem:justification_I} and Lemma \ref{lem:justification_II}. Therefore, $F(\xi,\lam)$ is $C^p$ with $p\geq 3$.

Next, we verify that assumptions (1) to (4) in Theorem \ref{thm:Crandall} hold for $F_\xi(0,\lam)$. Firstly, (1) obviously holds since these trivial solutions correspond to the symmetry solutions solved in Section \ref{sec:Symmetric solutions}. Regarding assumptions (2) and (3), recall that $X_1^{3+\alpha}$ and $X_1^{2+\alpha}$ share the same set of basis functions, and $F_{\xi} (0,\lam)$ as a bounded linear operator is characterized by the eigenvalue problem \eqref{eqn: eigenprob} with real, distinct, and nondegenerated eigenvalues. Thus, to verify (2) and (3), it is sufficient for us to check that our choice of $\lam_0^l$ ensures: 
\begin{equation}
\label{eqn: degenerateofE}
    E(\lam_0^l,j)\neq 0\quad\text{for any }\, j\neq l;\quad\text{and}\quad E(\lam_0^l,l)=0.
\end{equation}
Based on Proposition \ref{rem:F1} and the way chosen $\lam_0^l$ (recall \eqref{eqn:bifurcation point}), condition \eqref{eqn: degenerateofE} is indeed valid. Finally, for assumption (4), it suffices for us to show $\partial_{\lam}E(\lam_0^l,l)\neq 0$. By using condition \eqref{eqn:bifurcation point}, expression \eqref{eqn:partiallamE} yields,
\begin{align*}
\partial_{\lam}E(\lam_0^l,l)
&=\frac{G_0 c_B}{2\sqrt{\lam_0^l}}\left(\frac{l+1}{\sqrt{\lam_0^l}(\sqrt{\lam_0^l}+1)^2}-\frac{l+\sqrt{1+l^2}}{\sqrt{\lam_0^l+l^2}(\sqrt{\lam_0^l+l^2}+\sqrt{1+l^2})^2}\right).
\end{align*}
Note that the above value represents the slope at the intersection point of the curve with the horizontal axis in the left graph of Figure 1, thus it is always positive, and we conclude that
 \begin{equation}
     \left[F_{\xi \lam}(0,\lambda_0^l)\right]\cos{ly}=\left[\partial_{\lam}E(\lam_0^l,l)\right]\cos{ly}\notin \text{Img}\left[F_{\xi}(0,\lambda_0^l)\right].
 \end{equation}
By now, we have finished verifying all the assumptions in the Crandall-Rabinowitz theorem. Therefore, $(0,\lam_0^l)$ is a bifurcation point to \eqref{eqn: F=0} and generates a nontrivial solution branch. 
\end{proof}

\section*{Acknowledgments}
The work of Y.F. is supported by the National Key R\&D Program of China, Project Number 2021YFA1001200. J.-G.L is supported by NSF under award DMS-2106988. The work of Z.Z. is supported by the National Key R\&D Program of China, Project Number 2021YFA1001200, and NSFC grant number 12031013, 12171013. The authors thank Jiajun Tong (BICMR) and Xinyue Zhao (UTK) for helpful discussions and the anonymous referee for many helpful suggestions for improving the quality of this paper.

\appendix
\section{Incompressible limit}\label{ic}
To verify the validity of the Hele-Shaw type free boundary problem \eqref{eqn: mainmodel},  we provide the rigorous derivation of its weak from via the incompressible limit from the PME-type tumor growth model.  Recall that the PME type density equation of tumor growth writes,
\begin{equation}\label{tgn}
\begin{cases}
\partial_t\rho_m=\Delta\rho_m^m+G_0\rho_m c_m,&\\
\tau \partial_tc_m=\Delta c_m-\Psi(\rho_m,c_m),\quad &\\
c_m\to c_B,\quad \text{as  } |\textbf{x}|\to\infty,&\\
\end{cases}
\end{equation}
where $\Psi(\rho_m,c_m)$ is given by \eqref{eqn:PME_overall} and  we set $\chi:=\chi_{S}$ and $\Tilde{\chi}:=\chi_{S^c}$ for convenience. The pressure 
$P_m:=\frac{m}{m-1}\rho_m^{m-1}$ satisfies 
\begin{equation}\label{tgnp}
\partial_t P_m=(m-1)P_m(\Delta P_m+G_0c_m)+|\nabla P_m|^2.
\end{equation}
The initial data is given by 
\[(\rho_m(0),P_m(0)):=(\rho_{m,0},P_{m,0}),\quad \textbf{x}\in\mathbb{R}^n.\]
To align with the Hele-Shaw problem \eqref{eqn: mainmodel}, we consider $\tau=0$ for Eq.~\eqref{tgn}.

\vspace{2mm}
In the following, we first give some uniform in $m$ a priori estimates of the solution $(\rho_m,c_m,P_m)$. Then, we verify the incompressible limit through the weak solution and establish the complementarity relation via the viewpoint of obstacle problem inspired by~\cite{KimMellet2021}.
\begin{lem}[A priori estimates]\label{be}
Assume that $m>2$ and the initial data $(\rho_{m,0},P_{m,0})$ satisfy
\begin{equation*}
\begin{aligned}
&\|\rho_{m,0}\|_{L^1(\mathbb{R}^n)}\leq C,&&\|\Delta\rho_{m,0}^m+G_0\rho_{m,0} c_{m,0}\|_{L^1(\mathbb{R}^n)}\leq C,&&&\|\nabla \rho_{m,0}\|_{L^1(\mathbb{R}^n)}\leq C,\\
& P_{m,0}(\textbf{x})\leq C_0R_0^2h(\frac{\textbf{x}}{R_0}), &&\|\nabla P_{m,0}\|_{L^2(\mathbb{R}^n)}\leq C\\
\end{aligned}
\end{equation*}for some $R_0,C_0>0$ and $h(\textbf{x}):=\frac{1}{2n}(1-|\textbf{x}|^2)_+$. Then, the solution $(\rho_m,P_m,c_m)$ of  \eqref{tgn} satisfies the following uniform in $m$ a priori estimates as 
\begin{equation}\label{e1}0\leq c_{m}\leq c_B,\ (x,t)\in Q_T,\quad supp(P_{m}(t))\subset B_{R_T}(0),\ 0\leq t\leq T,\end{equation}
\begin{equation}\label{e2}\|P_{m}\|_{L^\infty(Q_T)}+\|\rho_{m}\|_{L^\infty(Q_T)}\leq C(T),\end{equation}
\begin{equation}\label{e3}\sup\limits_{0\leq t\leq T}[\|c_{m}-c_B\|_{L^2(\mathbb{R}^n)}+\|\nabla c_{m}\|_{L^2(\mathbb{R}^n)}+\|\partial_t c_m\|_{L^1(\mathbb{R}^n)}]\leq C(T),\end{equation}
\begin{equation}\label{e4}\sup\limits_{0\leq t\leq T}[\|\partial_t\rho_{m}\|_{L^1(\mathbb{R}^n)}+\|\nabla \rho_{m}\|_{L^1(\mathbb{R}^n)}]\leq C(T),\end{equation}
\begin{equation}\label{e5}\|\partial_t P_m\|_{L^1(Q_T)}+\|\nabla P_m\|_{L^2(Q_T)}\leq C(T),\end{equation}
where $R_T=R_0e^{\frac{C_0T}{n}}$.
\end{lem}
\begin{proof}
\emph{}
By the comparison principle,  we directly obtain 
\begin{equation}\label{est1}
0\leq c_m\leq c_B,\quad (\textbf{x},t)\in Q_T.
 \end{equation}
 Inspired by \cite{CKY2018}, let $C_0\geq G_0c_B$ and
\begin{equation*}
h(x):=\frac{1}{2n}(1-|\textbf{x}|^2)_+.
\end{equation*}
Then, for given $\phi:=C_0R^2(t)h(\frac{\textbf{x}}{R(t)})$, $R(t):=R_0e^{\frac{C_0t}{n}}$, it holds on the support of $\phi$ that 
\begin{equation*}
\begin{aligned}
\partial_t\phi-(m-1)\phi\underbrace{(\Delta \phi+G_0c_B)}_{\leq 0}-|\nabla\phi|^2\geq\frac{C_0R(t)}{n}(R'(t)-\frac{C_0R(t)}{n})=0.
\end{aligned}
\end{equation*}
Suppose that $P_{m,0}(x)\leq C_0R_0^2\phi(\frac{\textbf{x}}{R_0})$ with some $R_0>0$, it follows from the comparison principle that
\begin{equation*}
P_m(\textbf{x},t)\leq \phi(\textbf{x},t),\quad(\textbf{x},t)\in Q_T.
\end{equation*}
The above conclusion means
\begin{equation}\label{est2}
supp(P_{m})\subset B_{R_T}(0), \quad \|P_{m}\|_{L^\infty(Q_T)}\leq C(T),
\end{equation}
and further concludes for $m>2$ that 
\begin{equation}\label{rhoe}
supp(\rho_{m})\subset B_{R_T}(0), \quad \|\rho_{m}\|_{L^\infty(Q_T)}\leq C(T)
\end{equation}with $R_T=R_0e^{\frac{C_0T}{n}}$.
 We multiply~$\eqref{tgn}_2$ by $c_m-c_B$ and obtain
 \begin{equation*}
 \begin{aligned}
 \int_{\mathbb{R}^n}|\nabla c_m|^2dx+\int_{\mathbb{R}^n}[\lambda\chi \rho_m+\tilde{\chi}(1-\rho_m)_+]|c_m-c_B|^2dx\leq \int_{\mathbb{R}^n}\chi \rho_m c_B|c-c_B|dx.
 \end{aligned}
 \end{equation*}
Using the fact $\lambda\chi\rho_m+\tilde{\chi}(1-\rho_m)_+\geq \min\{\lambda\chi,\tilde{\chi}\}>0$, the estimate~\eqref{rhoe}, and H{\"o}lder's inequality yields 
 \begin{equation}\label{est3}
 \sup\limits_{0\leq t\leq T}\|c_m(t)-c_B\|_{L^2(\mathbb{R}^n)}+\sup\limits_{0\leq t\leq T}\|\nabla c_m(t)\|_{L^2(\mathbb{R}^n)}\leq C(T).
\end{equation}
 We use the equation $\eqref{tgn}_2$ and Kato's inequality, it holds  
 \begin{equation}\label{est4}
 0\leq \Delta |\partial_t c_m|-|\partial_t c_m|(\lambda\chi\rho_m+\tilde{\chi}(1-\rho_m)_+)+|\partial_t\rho_m|(G_0\chi c_m+\tilde{\chi}|c_m-c_B|), 
 \end{equation}
 which implies 
 \begin{equation}\label{ctrhot}
 \int_{\mathbb{R}^n}|\partial_t c_m|dx\leq C\int_{\mathbb{R}^n}|\partial_t\rho_m|dx,\ t>0.
 \end{equation}
 By means of Kato's inequality for $\eqref{tgn}_1$, we have 
 \begin{equation}\label{inet}
 \partial_t |\partial_t\rho_m|\leq \Delta |\partial_t\rho_m^m|+G_0c_B|\partial_t\rho_m|+G_0\rho_m|\partial_t c_m|.
 \end{equation}
 Taking \eqref{ctrhot} into consideration,  and integrating the above inequality \eqref{inet} on $[0,t]$ for any $0\leq t\leq T$, we have
 \begin{equation}\label{est5}
 \sup\limits_{0\leq t\leq T}\|\partial_t\rho_m\|_{L^1(\mathbb{R}^n)}+\sup\limits_{0\leq t\leq T}\|\partial_t c_m\|_{L^1(\mathbb{R}^n)}\leq C(T).
 \end{equation}
 Parallel to the proof of the above estimate, it holds 
 \begin{equation}\label{est6}
 \sup\limits_{0\leq t\leq T}\|\nabla \rho_m\|_{L^1(\mathbb{R}^n)}\leq C(T).
 \end{equation}
 We multiply the inequality equation~$\eqref{inet}_1$ by $\theta(x)\geq 0$ with $-\Delta \theta =\chi_{B_{R_T+1}}$ and integrate on $Q_T$, then it yields 
 \begin{equation*}
 \begin{aligned}
 \|\partial_t \rho_m^m\|_{L^1(Q_T)}\leq &\|\theta\|_{L^\infty(\mathbb{R}^n)}[(\|\partial_t\rho_m(0) \|_{L^1(\mathbb{R}^n)}+\|\partial_t\rho_m(T) \|_{L^1(\mathbb{R}^n)})\\
 &+\|G_0c_B|\partial_t\rho_m|+G_0\rho_m|\partial_t c_m|\|_{L^1(Q_T)}]\leq C(T).
 \end{aligned}
 \end{equation*}
 Hence, we obtain
 \begin{equation}\label{est7}
 \|\partial_t P_m\|_{L^1(Q_T)}\leq \frac{m}{2^{m-2}}\|\partial_t\rho_m\|_{L^1(Q_T)}+2\|\partial_t\rho_m^m\|_{L^1(Q_T)}\leq C(T).
 \end{equation}
 In addition, it holds by integrating~\eqref{tgnp} on $Q_T$ that 
\begin{equation}\label{est8}
\begin{aligned}
\|\nabla P_m\|_{L^2(Q_T)}&\leq \frac{m-1}{m-2}G_0c_B\|P_m\|_{L^1(Q_T)}+\frac{1}{m-1}[\|P_{m,0}\|_{L^1(\mathbb{R}^n)}+\|P_m(T)\|_{L^1(\mathbb{R}^n)}]\\
&\leq C(T).
\end{aligned}
\end{equation}
Combining \eqref{est1},\eqref{est2},\eqref{est3},\eqref{est4},\eqref{est5},\eqref{est6},\eqref{est7}, and \eqref{est8}, we complete the proof.
\end{proof}

Based on the basic  a priori estimates in Lemma~\ref{be}, we derive some convergence results for the solution $\rho_m, P_m, c_m$ in $m$, and then further prove both the incompressible limit and the complementarity relation.
\begin{thm}
Under the same initial assumptions of Lemma~\ref{be}, there exist a pair of functions  $(\rho_\infty, P_\infty, c_\infty)$, satisfying  $\rho_{\infty}\in L^1(Q_T)\cap L^\infty(Q_T)\cap BV(Q_T)$, $ P_\infty\in L^1(Q_T)\cap L^\infty(Q_T)\cap BV(Q_T)\cap L^2(0,T;H^1(\mathbb{R}^n))$, $c_\infty-c_B\in L^2(0,T;H^1(\mathbb{R}^n))\cap BV(Q_T)$ with $0\leq c_\infty\leq c_B$, such that, after extracting subsequences, as $m\to \infty$, it holds
\begin{align}
&\rho_m\to \rho_\infty,&&\text{in }L^p(Q_T)\text{ with }\ 1\leq p<\infty,\label{c1}\\
&P_m\to P_\infty,&&\text{in }L^p(Q_T)\text{ with } 1\leq p<\infty,\label{c2}\\
&c_m\to c_\infty,&& \text{in }L^p_{loc}(Q_T)\text{ with } 1\leq p<\infty,\label{c3}\\
&(1-\rho_m)_+\to (1-\rho_\infty)_+,&& \text{in }L^p(Q_T)\text{ with } 1\leq p<\infty.\label{c4}
\end{align}
Moreover, the limit $(\rho_\infty,c_\infty,P_\infty)$ satisfies a Hele-Shaw type system as 
\begin{equation}\label{tgnhs}
\begin{cases}
\partial_t\rho_\infty=\Delta P_\infty+G_0\rho_\infty c_\infty,&\text{in }\mathcal{D}'(Q_T),\\
\Delta c_\infty=\Psi(\rho_\infty,c_\infty),\quad &\text{in }\mathcal{D}'(Q_T),\\
c_\infty\to c_B,\quad \text{as  } |\textbf{x}|\to\infty.&\\
0\leq \rho_\infty\leq 1,\quad P_\infty(1-\rho_\infty)=0,& \text{a.e. in }Q_T,
\end{cases}
\end{equation}
with the initial data $\rho_{\infty,0}\in L^1(\mathbb{R}^n)\cap L^\infty(\mathbb{R}^n)\cap BV(\mathbb{R}^n)$. Moreover, the complementarity relation holds as 
\begin{equation}\label{complementarity}
P_{\infty}(\Delta P_{\infty}+G_0c_\infty)=0\quad \text{in }\mathcal{D}'(\mathbb{R}^n\times (0,\infty)).
\end{equation}
\end{thm}
\begin{proof}\emph{Step 1, proof of \eqref{tgnhs}.}
Based on verified estimates~\eqref{e1}-\eqref{e5}, then \eqref{c1}-\eqref{c3} hold by the compactness embedding. Furthermore, \eqref{c1} naturally proves~\eqref{c4}.   Taking \eqref{c1}-\eqref{c4} into account, hence $\eqref{tgnhs}_{1-3}$ holds in the sense of distribution. Since $0\leq P_m\leq C(T)$, it holds by passing to limit that 
 \begin{equation}\label{k1}0\leq \rho_\infty \leq 1,\ \text{a.e. in }Q_T.\end{equation}
In addition, we have $P_m^{\frac{m}{m-1}}=(\frac{m}{m-1})^{\frac{1}{m-1}}\rho_m P_m$, which concludes after taking the limit in $m$ that 
  \begin{equation}\label{k2}P_\infty(1-\rho_\infty)=0,\ \text{a.e. in }Q_T.\end{equation}
Hence, \eqref{k1}-\eqref{k2} prove $\eqref{tgnhs}_{4}$.

\vspace{2mm}

\noindent \emph{Step 2, proof of \eqref{complementarity}.} We set a functional space 
\begin{equation*}
E_t:=\{v\in H^1(\mathbb{R}^n)\cap L^1(\mathbb{R}^n)\ |\ v\geq 0,\langle v,1-\rho_\infty(t)\rangle_{H^1,H^{-1}}=0\}.
\end{equation*}
Let $v$ be a function in $E_{t_0}$ with $t_0>0$, we use the density equation~\eqref{tgn}  and  the  pressure equation \eqref{tgnp} to obtain
\begin{equation}
\begin{aligned}
\int_{\mathbb{R}^n}&\nabla P_{m}\cdot\nabla P_{m}-\rho_{m}\nabla P_{m}\cdot  v+G_0c_m(v-P_{m})dx\\
=&-\frac{1}{m-1}\big[\frac{d}{dt}\int_{\mathbb{R}^n}P_{m}dx-\int_{\mathbb{R}^n}|\nabla P_{m}|^2dx\big]+\frac{d}{dt}\int_{\mathbb{R}^n}v\rho_{m}dx\\
&+\int_{\mathbb{R}^n}G_0c_mv(1-\rho_{m})dx.
\end{aligned}
\end{equation}
Integrating on $(t_0,t_0+\delta)$ for any $\delta>0$, we have
\begin{equation}
\begin{aligned}
\int_{t_0}^{t_0+\delta}&\hskip-4pt\int_{\mathbb{R}^n}\nabla P_{m}\cdot\nabla P_{m}-\rho_{m}\nabla P_{m}\cdot  v+G_0c_m(v-P_{m})dxdt\\
=&-\frac{1}{m-1}\big[\int_{\mathbb{R}^n}(P_{m}(t_0+\delta)-P_{m}(t_0))dx-\int_{t_0}^{t_0+\delta}\hskip-4pt\int_{\mathbb{R}^n}|\nabla P_{m}|^2dxdt\big]\\
&+\int_{\mathbb{R}^n}v(\rho_{m}(t_0+\delta)-\rho_{m}(t_0))dx
+\int_{t_0}^{t_0+\delta}\hskip-4pt\int_{\mathbb{R}^n}G_0c_mv(1-\rho_{m})dxdt.
\end{aligned}
\end{equation}
Thanks to the  lower semi-continuity of $L^2$-norm, it holds 
\begin{equation}\label{liminf}
\begin{aligned}
&\int_{t_0}^{t_0+\delta}\hskip-4pt\int_{\mathbb{R}^n}\nabla P_{\infty}\cdot\nabla P_{\infty}-\nabla P_{\infty}\cdot  v+G_0c_\infty(v-P_{\infty})dxdt\\
&\leq \liminf\limits_{m\to\infty}\int_{t_0}^{t_0+\delta}\hskip-4pt\int_{\mathbb{R}^n}\nabla P_{m}\cdot\nabla P_{m}-\rho_{m}\nabla P_{m}\cdot  v+G_0c_m(v-P_{m})dxdt\\
&=\liminf\limits_{m\to\infty}\int_{\mathbb{R}^n}v(\rho_{m}(t_0+\delta)-\rho_{m}(t_0))dx+\int_{t_0}^{t_0+\delta}\hskip-4pt\int_{\mathbb{R}^n}G_0c_\infty v(1-\rho_\infty)dxdt\\
&\leq \liminf\limits_{m\to\infty}\int_{\mathbb{R}^n}v(\rho_{m}(t_0+\delta)-\rho_{m}(t_0))dx+G_0c_B\int_{t_0}^{t_0+\delta}\hskip-4pt\int_{\mathbb{R}^n}v(1-\rho_\infty)dxdt.
\end{aligned}
\end{equation}
We note that 
\begin{equation*}
\frac{d}{dt}\int_{\mathbb{R}^n}\rho_{m}vdx=-\int_{\mathbb{R}^n}\rho_{m}\nabla P_{m}\cdot \nabla vdx+ \int_{\mathbb{R}^n}G_0c_m\rho_{m}vdx.
\end{equation*}
Then, it follows
\begin{equation*} 
|\frac{d}{dt}\int_{\mathbb{R}^n}\rho_{m}vdx|\leq C\int_{\mathbb{R}^n}|\nabla P_{m}| |\nabla v|dx+G_0c_B \int_{\mathbb{R}^n}\rho_{m}vdx.
\end{equation*}

The first term on the right-hand side is bounded in $L^2(0,T)$, and the second term is also bounded in $L^\infty(0,T)$. We deduce that the function $t\to \int_{\mathbb{R}^n}\rho_{m}vdx$ is bounded in $H^1(0,T)\subset C^{1/2}(0,T)$ and therefore converges uniformly in $[0,T]$. Since $\int_{\mathbb{R}^n}v\rho_{m}dx$ converges to $\int_{\mathbb{R}^n}v\rho_\infty dx$ in $\mathcal{D}'(\mathbb{R}_+)$, we have
\begin{equation*}
\int_{\mathbb{R}^n}v(\cdot)\rho_{m}(\cdot,t)dx\to \int_{\mathbb{R}^n}v(\cdot)\rho_{\infty}(\cdot,t)dx\text{ locally uniformly in }\mathbb{R}_+.
\end{equation*}
Consequently,
\begin{equation}\label{rhomo}
\begin{aligned}
\liminf\limits_{m\to\infty}&\int_{\mathbb{R}^n}v(\cdot)(\rho_{m}(\textbf{x},t_0+\delta)-\rho_{m}(\textbf{x},t_0))dx\\
=&\int_{\mathbb{R}^n}v(\cdot)(\rho_\infty(\textbf{x},t_0+\delta)-\rho_\infty(\textbf{x},t_0))dx\\
=&\int_{\mathbb{R}^n}v(\cdot)(\rho_\infty(\textbf{x},t_0+\delta)-1)dx\leq 0.
\end{aligned}
\end{equation}
We insert~\eqref{rhomo} into~\eqref{liminf}, it holds
\begin{equation*}
\begin{aligned}
\frac{1}{\delta}\int_{t_0}^{t_0+\delta}\hskip-4pt\int_{\mathbb{R}^n}\frac{|\nabla P_{\infty}|^2}{2}-G_0c_\infty P_{\infty} dxdt\leq& \frac{1}{\delta}\int_{t_0}^{t_0+\delta}\hskip-4pt\int_{\mathbb{R}^n}\frac{|\nabla v|^2}{2}-G_0c_\infty v dxdt\\
&+\frac{G_0c_B}{\delta}\int_{t_0}^{t_0+\delta}\langle v,1-\rho_\infty\rangle_{H^1,H^{-1}}dt.
\end{aligned}
\end{equation*}
Due to the fact $P_\infty \in BV(Q_T)$, the trace theorem supports $P_\infty^+=P_\infty$ for the trace $P_\infty^+$ of $P_\infty$. Hence, let $\delta \to 0^+$, we get 
\begin{equation}\label{varine}
\begin{aligned}
&\int_{\mathbb{R}^n}\frac{|\nabla P_{\infty}|^2}{2}-G_0c_\infty P_{\infty} dx \leq \int_{\mathbb{R}^n}\frac{|\nabla v|^2}{2}-G_0c_\infty v dx,
\end{aligned}
\end{equation}
where $\frac{G_0c_B}{\delta}\int_{t_0}^{t_0+\delta}\langle v,1-\rho_\infty\rangle_{H^1,H^{-1}}dt\to \langle v,1-\rho_\infty(t_0)\rangle _{H^1,H^{-1}}=0$ as $t\to 0^+$ is used since $1-\rho_\infty\in C([0,T],H^{-1})$. We can conclude that $P_\infty(t)$ is a global minimizer in $E_t$ a.e. $t>0$.

\vspace{1mm}

Given a test function $\varphi \in C_0^\infty(\mathbb{R}^n\times (0,\infty)) $, we take $v_\epsilon=P_{\infty}+\epsilon P_{\infty} \varphi=P_{\infty}(1+\epsilon\varphi)$ with $|\epsilon|\ll 1$ so that $1+\epsilon\varphi\geq0$.  $P_{\infty}(1+\epsilon\varphi)\in E_t$ holds evidently. Due to~\eqref{varine}, we have
\begin{equation*}
\frac{d}{d\epsilon}\Big\vert_{\epsilon=0}\big[\int_{\mathbb{R}^n}\frac{|\nabla v_\epsilon|}{2}-G_0c_\infty v_\epsilon dx\big]=0,
\end{equation*}
which yields 
\begin{equation*}
\int_{\mathbb{R}^n}\nabla P_{\infty}(t)\cdot\nabla (P_{\infty}(t)\varphi)-G_0c_\infty P_{\infty}(t)\varphi dx=0,\quad\text{a.e. }t\in\mathbb{R}_+.
\end{equation*}
Hence, the complementarity relationship
\eqref{complementarity} holds in the sense of distribution.
\end{proof}

\section{Justification of the expansion}
\label{sec:Justification of the expansion}
\label{Justification of the expansion}
We devote this section to the proof of Lemma \ref{lem:justification_I} and Lemma \ref{lem:justification_II}. 
To begin with, recall that $\Omega$ is the tube-like domain defined in \eqref{eqn:infinity_tube}. $\Omega_0$ and $\Omega_{\xi}$ correspond to the unperturbed and perturbed tumor region, respectively. For concision, we denote the complementary sets as
\begin{equation}
\Omega_0^{\text{c}}=\Omega\setminus\Omega_0,\qquad
\Omega_{\xi}^{\text{c}}=\Omega\setminus\Omega_{\xi}.
\end{equation}
Now, we provide the proof of Lemma \ref{lem:justification_I} as follows. 
\begin{proof}[proof of Lemma \ref{lem:justification_I}]
Note that if we denote $c^{\delta}=c-c_0$, then it satisfies
\begin{subequations}
\begin{alignat}{2}
   -\Delta c^{\delta}+\lambda c^{\delta}&=0,&\qquad\text{in}\quad\Omega_{\xi}\cap\Omega_0:=\Omega_1;\\ 
   -\Delta c^{\delta}+\lambda c^{\delta}&=(1-\lambda)c_0-c_B,&\qquad\text{in}\quad\Omega_{\xi}\cap\Omega_0^{\text{c}}:=\Omega_2;\\ 
    -\Delta c^{\delta}+ c^{\delta}&=(\lambda-1)c_0+c_B,&\qquad\text{in}\quad\Omega_{\xi}^{\text{c}}\cap\Omega_0:=\Omega_3;\\ 
    -\Delta c^{\delta}+ c^{\delta}&=0,&\qquad\text{in}\quad\Omega_{\xi}^{\text{c}}\cap\Omega_0^{\text{c}}:=\Omega_4.
\end{alignat}
\end{subequations}
Write them in a single equation, we get
\begin{equation}
   -\Delta c^{\delta}+\left(\lambda\cdot\chi_{\Omega_{\xi}}+\chi_{\Omega_{\xi}^{\text{c}}}\right)c^{\delta}=\left((1-\lambda)c_0-c_B\right)\cdot\left(\chi_{\Omega_2}-\chi_{\Omega_3}\right),\qquad\text{in}\quad\Omega.
\end{equation}
Observe the fact that $\lambda\chi_{\Omega_{\xi}}+\chi_{\Omega_{\xi}^{\text{c}}}$ can be treated as a function in $L^{\infty}(\Omega)$, $c_0$ has already been solved explicitly in $\Omega$. Furthermore, the areas $\vert\Omega_2\vert$ and $\vert\Omega_3\vert$ are bounded by $\epsilon\Vert\Tilde{\xi}\Vert_{C^{3+\alpha}(\mathbb{R})}$. Then, the classical $W^{2,q}$ estimate of elliptic equations and Sobolev embedding theory together yield the first inequality in Lemma \ref{lem:justification_I}. More precisely, for any $q>2$ and $\alpha=1-2/q$ one has
\begin{equation*}
\Vert c^{\delta}\Vert_{C^{1+\alpha}(\Omega)}
\leq \Vert c^{\delta}\Vert_{W^{2,q}(\Omega)}
\leq \Vert\left((1-\lambda)c_0-c_B\right)\cdot\left(\chi_{\Omega_2}-\chi_{\Omega_3}\right)\Vert_{L^q(\Omega)}
\leq C\vert\epsilon\vert\Vert\Tilde{\xi}\Vert_{C^{3+\alpha}(\mathbb{R})}.
\end{equation*}
Finally, send $q\rightarrow\infty$ to complete the proof.

For the second inequality in Lemma \ref{lem:justification_I}, one can easily write down the equation for $p^{\delta}=p-p_0$, that is
\begin{equation}
-\Delta p^{\delta}=G_0 c^{\delta},\quad\text{in}\quad\Omega,
\end{equation}
 Thus, by using Schauder estimate one has
\begin{align}
\Vert p^{\delta}\Vert_{C^{3+\alpha}(\Omega)}
\leq G_0\Vert c^{\delta}\Vert_{C^{1+\alpha}(\Omega)}
\leq C\vert\epsilon\vert\Vert\Tilde{\xi}\Vert_{C^{3+\alpha}(\mathbb{R})}.
\end{align}
\end{proof}
Next, observe the fact that $(c^{\text{(i)}}, c^{\text{(o)}}, p^{\text{(i)}}, p^{\text{(o)}})$ are defined
in $\Omega_{\xi}$ or $\Omega_{\xi}^c$, respectively. However, the first-order terms $(c_1^{\text{(i)}}, c_1^{\text{(o)}}, p_1^{\text{(i)}}, p_1^{\text{(o)}})$ are only defined in $\Omega_0$ or $\Omega_0^c$. Therefore, we need to transform them to $\Omega_{\xi}$ or $\Omega_{\xi}^c$ by Hanzawa transformation $\calH_{\xi}$, which is defined as follows:
\begin{equation}
\label{eqn:Hanzawa}
    (x,y)=\calH_{\xi}(x',y')=(x'+\calI (x')\epsilon\Tilde{\xi},y'),
\end{equation}
where $\calI\in C^{\infty}$ satisfies
\begin{equation*}
   \calI(\zeta)=\left\{\begin{array}{cc}
         &0,\quad\text{if}\quad\vert\zeta\vert\geq\frac{3}{4}\delta, \\ &1,\quad\text{if}\quad\vert\zeta\vert<\frac{1}{4}\delta,
    \end{array}\right., \quad\text{with}\quad\left\vert\frac{d^k\calI}{d\zeta^k}\right\vert<\frac{C}{\delta^k},
\end{equation*}
where $\delta$ is a small positive scalar. Thus, $\calH_{\xi}$ maps $\Omega_0$ to $\Omega_{\epsilon}$, and maps $\Omega_0^c$ to $\Omega_{\epsilon}^c$. We denote
\begin{subequations}
\label{eqn: HZ transform}
\begin{alignat}{2}
    \hat{c}_1^{\text{(i)}}(x,y;\epsilon\Tilde{\xi})&=c_1^{\text{(i)}}(\calH_{\xi}^{-1}(x,y);\epsilon\Tilde{\xi}),\qquad
    \hat{c}_1^{\text{(o)}}(x,y;\epsilon\Tilde{\xi})&=c_1^{\text{(o)}}(\calH_{\xi}^{-1}(x,y);\epsilon\Tilde{\xi}),\\
    \hat{p}_1^{\text{(i)}}(x,y;\epsilon\Tilde{\xi})&=p_1^{\text{(i)}}(\calH_{\xi}^{-1}(x,y);\epsilon\Tilde{\xi}),\qquad
    \hat{p}_1^{\text{(o)}}(x,y;\epsilon\Tilde{\xi})&=p_1^{\text{(o)}}(\calH_{\xi}^{-1}(x,y);\epsilon\Tilde{\xi}).
\end{alignat}
\end{subequations}
Then, we further define
\begin{align}
\label{eqn:twohat}
\hat{c}_1(x,y;\epsilon\Tilde{\xi})&:=\hat{c}_1^{\text{(i)}}(x,y;\epsilon\Tilde{\xi})+\hat{c}_1^{\text{(o)}}(x,y;\epsilon\Tilde{\xi}),\\
\hat{h}_1(x,y;\epsilon\Tilde{\xi})&:=\hat{h}_1^{\text{(i)}}(x,y;\epsilon\Tilde{\xi})+\hat{h}_1^{\text{(o)}}(x,y;\epsilon\Tilde{\xi}).
\end{align}
Now, we turn to the proof of Lemma \ref{lem:justification_II}. The detail of the proof is cumbersome, but the idea is quite simple and in the same manner as the proof of the Lemma \ref{lem:justification_I}. Therefore, we only provide a sketch of it.
\begin{proof}[proof of Lemma \ref{lem:justification_II}]
The proof is similar to that of Lemma \ref{lem:justification_I}. Denote $c^{\delta}:= c-c_0-\epsilon \hat{c}_1$ and similarly for $p^{\delta}:=p-p_0-\epsilon \hat{p}_1$, where $\hat{c}_1$ and $\hat{p}_1$ are defined in the same manner as \eqref{eqn:inandout}. Then, one can write the equation for $c^{\delta}$ on the whole $\Omega$. Then, employ $W^{2,q}$ estimate of the elliptic equations and embedding theory to obtain the nutrient estimate first, as we did in Lemma \ref{lem:justification_I}. However, to do this, one needs to compute the first and second derivatives of $\hat{c}_1$ with respect to $(x,y)$, which further requires us to consider the change of variables induced by the Hanzawa transformation. This process is cumbersome, but is standard. Therefore, we refer the reader to Theorem 4.5 in \cite{lu2022bifurcation} for a similar proof. Once the nutrient estimate is obtained, the pressure estimate can be obtained using the Schauder estimate in the same way as in the proof of Lemma \ref{lem:justification_I}.
\end{proof}
\bibliography{tumorbibfile}
\bibliographystyle{plain}
\end{document}